\documentclass{article}
\usepackage{amsmath}
\usepackage{amsfonts}
\usepackage{amssymb}
\usepackage{setspace}
\usepackage{ntheorem}
\usepackage[numbers,sort&compress]{natbib}
\usepackage{graphicx}
\usepackage{geometry}
\usepackage[colorlinks,linkcolor=blue,citecolor=blue]{hyperref}
\pdfoutput=1

\begin{document}

\newtheorem{remark}{Remark}[section]
\newenvironment{proof}{{\noindent\it Proof.}\quad}{\hfill $\square$\par}
\newtheorem{theorem}{Theorem}[section]
\newtheorem{lemma}{Lemma}[section]

\numberwithin{equation}{section}

\title{\Large   Accelerated Primal Dual Method for a Class of Saddle Point Problem with Strongly Convex Component
} 
\author{Zhipeng Xie\qquad Jianwen Shi
	\thanks{pxie@hqu.edu.cn, College of Computer Science and Technology, Huaqiao University, Xiamen,Fujian,361021,PRChina}
	}

\date{}

\maketitle
\begin{spacing}{1.5}
   
	\rmfamily
	\noindent \textbf{Abstract:} This paper presents a simple primal dual method named DPD which is a flexible framework for a class of saddle point problem with or without strongly convex component. The presented method has linearized version named LDPD and exact version EDPD. Each iteration of DPD updates sequentially the dual and primal variable via simple proximal mapping and refines the dual variable via extrapolation. Convergence analysis with smooth or strongly convex primal component recovers previous state-of-the-art results, and that with strongly convex dual component attains full acceleration $O(1/k^2)$ in terms of primal dual gap. Total variation image deblurring on Gaussian noisy or Salt-Pepper noisy image demonstrate the effectiveness of the full acceleration by imposing the strongly convexity on dual component.\\\\
\noindent \textbf{Keywords:}   saddle point problem, strongly convex component, multi-step gradient method, full acceleration, total variation image deblurring. \\\\
\noindent\textbf{Mathematics Subject Classification:} 90C06, 90C25, 90C46, 90C47

\section{Introduction}

\subsection{Saddle Point Problem with Strongly Convex Component}
In this paper, we consider the saddle point problem (SPP)\cite{Xu2016Accelerated,Chambolle2011A,Chen2013Optimal,Ouyang2014An,Goldstein2014Fast}
in the form of
\vspace{-2mm}
\begin{equation}
\mathop {\min }\limits_x \mathop {\max }\limits_y \left\{ {L(x,y) = f(x) + \left\langle {Ax,y} \right\rangle  - g(y)} \right\}
\label{eq1.1}
\end{equation}

\noindent where $f(x)$ denotes primal component,  $g(y)$ dual component and $<Ax,y>$ the bilinear inner product term. In the case of $g(y)$ being conjugate of $G(x)$ , Model (\ref{eq1.1}) can be converted into 
composite optimization as
\vspace{-2mm}
\begin{equation}
\mathop {Min}\limits_{x \in X} \begin{array}{*{20}{c}}
{}
\end{array}f(x) + G(Ax)
\end{equation}

\noindent which covers many sophisticated applications including total variation image processing
\cite{Xu2016Accelerated,Chambolle2011A,Chen2013Optimal} and overlapping group lasso optimization\cite{Chen2013Optimal}\cite{Lei2013Efficient}.
 In this paper, we mainly focus on the case where primal component $f(x)$ or dual component $g(y)$ is strongly convex. For instance, total variation image denoising can be written as
\begin{equation}
\mathop {Min}\limits_{x \in X} \begin{array}{*{20}{c}}
{}
\end{array}(\alpha /2)||x - b|{|^2} + ||Dx|{|_{2,1}}
\end{equation}

\noindent which contains strongly convex primal component $f(x) = \alpha ||x - b|{|^2}/2$ . Image deblurring with smoothed total variation can be written as

\begin{equation}
\mathop {Min}\limits_x \frac{\alpha }{2}||Ax - b|{|^2} + \mathop {Max}\limits_y \left\langle {Dx,y} \right\rangle  - \frac{{{\mu _g}}}{2}||y|{|^2}
\end{equation}

\noindent which uses strongly convex dual component $g(y) = {\mu _g}||y|{|^2}/2$.

\subsection{ Recent Advances in Primal-Dual Methods}
\indent Nowadays first order methods become popular. Various efforts \cite{Amir2009Fast,Beck2009AFI,Beck2014A,Chambolle2011A,Chen2013Optimal,Goldstein2014Fast,
	Lei2013Efficient,Ouyang2014An,Xu2016Accelerated,Yang2012An}  are devoted to achieve better performance for different problems. Multi-step gradient acceleration with convergence rate $O(1/k^2)$ was proposed firstly by Nesterov in \cite{Nesterov1983A,Nesterov1988A,Nesterov2005A} , which encompasses three different methods whose unified analysis can be found in Paul Tseng’s work\cite{Tseng2008A}. Since the prevail of FISTA\cite{Beck2009AFI} by Beck and Teboulle in machine learning, more and more first order methods \cite{Chen2013Optimal,Ouyang2014An,Goldstein2014Fast,Lei2013Efficient,Goldfarb2013Fast,Becker2009NESTA} adopts different type of Nesterov accelerating methods. Note that directly applying multi-step gradient method to overlapping group lasso or total variation image deblurring usually demands two loops iteration: outer loop for approximating smooth objective component and inner loop for total variation image denoising sub-problem.

Meanwhile, with single layer iteration structure and fast Fourier transformation, Alternating Direction Methods of Multiplier (ADM)\cite{Deng2016A,He2012A,He2012Convergence,Chan2013Constrained,Ng2011Inexact,Boyd2010} become popular in image processing and machine learning partly due to proof on $O(1/k)$ convergence rate by He and Yuan \cite{He2012A}, linear convergence by Deng and Yin \cite{Deng2016A} and tutorial and analysis by Boyd \cite{Boyd2010}. Nevertheless, linearization and its acceleration are less addressed in the classic ADM methods.

On the other hand, since classic primal-dual proximal methods directly work on saddle point model, it draws continuous attentions\cite{Chambolle2011A,Chen2013Optimal,Ouyang2014An,Goldstein2014Fast}  in both analysis and applications. The $O(1/k)$ convergence rate for exact primal-dual method was presented by Chambolle and Pock in \cite{Chambolle2011A} where acceleration via strongly convex component is illustrated in terms of distance between primal-dual iterations and saddle points. In \cite{Goldstein2014Fast} ,Goldstein presented fast ADM with $O(1/k^2)$ convergence rate in terms of dual objective, which requires both objective functions being strongly convex . Nevertheless, the above two methods did not address the linearization of objective. Recently, using second type  of multi-step gradient method, Chen, Lan, and Ouyang propose optimal primal-dual method \cite{Chen2013Optimal}  for a class of weakly convex deterministic or stochastic problems with convergence rate $O(1/k^2)+O(1/k)$.  In \cite{Ouyang2014An}, Ouyang, Chen, Lan and Pasiliao present accelerated linearized ADM with convergence rate $O(1/k^2)+O(1/k)$. Very lately, Xu Yangyang \cite{Xu2016Accelerated} provides Linearized ADM with full acceleration $O(1/k^2)$ in terms of primal objective and residual. The full accelerations in \cite{Xu2016Accelerated} utilizes smoothness and strongly convexity of primal component and seems escaping successfully the gravity of Nesterov’s multi-step gradient method. Convergence analysis in \cite{Xu2016Accelerated}  is very intuitive and interesting.
\subsection{ Motivation and Contribution}
Motivated by the fine arts \cite{Xu2016Accelerated} on full acceleration via strongly convexity of primal component,  this paper presents a primal-dual method named DPD which allows to utilize the strongly convexity of both primal and dual component. The presented method updates sequentially dual-primal variable by gradient projections and corrects dual variable by extrapolation, hence avoiding extra-gradient projection for second dual update as classical extra-gradient method \cite{Nemirovski2004A,Bonettini2014An}. DPD allows linearization of primal component, partial acceleration and full acceleration. Its convergence analysis recovers state-of-the-art result for a class of weak convex \cite{Chen2013Optimal,Yang2012An,Nemirovski2004A} or strongly convex problem \cite{Xu2016Accelerated}.\\
\indent \emph{Contribution of presented DPD mainly consists of two points.\\}
\indent(1) Linearized DPD (LDPD) with strongly convex dual component achieves full acceleration $O(1/k^2)$ on Gaussian noisy image deblurring. It converges faster than treating as weakly convex problem with partial acceleration $O(1/k^2)+O(1/k)$.\\
\indent(2) Exact DPD (EDPD) with strongly convex dual component attains full acceleration   
$O(1/k^2)$ on Salt-Pepper noisy image deblurring, and converges faster than treating as weakly convex non-smooth problem with the rate $O(1/k)$.

\subsection{Paper Layout}
The rest of this paper is organized as follows. Section 2 presents Iteration of Linearized DPD(LDPD). Convergence analysis for LDPD is presented in section 3 and it covers multi-step gradient LDPD with partial acceleration, full acceleration by strongly convex dual component and that by primal component. Section 4 presents Exact DPD(EDPD) without linearization and develops analysis on full acceleration by strongly convex primal or dual component. Section 5 discusses implementation of LDPD and EDPD with strong convex dual component in noisy image deblurring. Section 6 illustrates acceleration performance of DPD in noisy Image deblurring. Finally is a brief discussion.

\section{Linearized DPD(LDPD) }\label{section2}
For LDPD, we assume primal component $f(x)$ is differentiable with gradient lipschitz constant $L_f$ , then linearize $f(x)$ and use multi-step gradient method to update primal variable $x$.
\newpage
\subsection{Iteration of LDPD}
Let ${x_1} \in X,{y_1} \in Y,{\theta _t} \in (0,1],{\bar x_1} = {x_1},{\hat y_1} = {y_1},\forall N $ , iteration of LDPD proceeds as\\
\noindent\rule{17cm}{0.05em}\\
	\textbf{For} $t=1,2,...,N$ 
	\vspace{-5mm}
	\begin{align}
	{\hat x_t\quad}&= (1 - {\theta _t}){\bar x_t} + {\theta _t}{x_t}\label{eq2.1}\\
	{x_{t + 1}}&= \arg {\min _x}{\rm{||}}x - [{x_t} - {\eta _t}(\nabla f({\hat x_t}) + {A^T}{\hat y_t})]{\rm{|}}{{\rm{|}}^2}\label{eq2.2}\\
	{\bar x_{t + 1}}&= (1 - {\theta _t}){\bar x_t} + {\theta _t}{x_{t + 1}}\label{eq2.3}\\
	{y_{t + 1}}&= \arg {\min _y}||y - ({y_t} + {\tau _t}A{x_{t + 1}})|{|^2}/(2{\tau _t}) + g(y)\label{eq2.4}\\
	{\hat y_{t + 1}}&= ({y_{t + 1}} - {y_t}){\alpha _{t + 1}} + {y_{t + 1}}\label{eq2.5}\\
	{\bar y_{t + 1}}&= (1 - {\theta _t}){\bar y_t} + {\theta _t}{y_{t + 1}}\label{eq2.6}
	\end{align}
\textbf{End For}\\
	\rule{17cm}{0.05em}\\
		\vspace{-5mm}

Setting  $\theta_t=1$ would change multi-step gradient LDPD (\ref{eq2.1})-(\ref{eq2.6}) into single-step gradient which uses the following parameter
\begin{equation}
{\hat x_t} = {x_t},\nabla f({\hat x_t}) = \nabla f({x_t}),{\bar x_{t + 1}} = {x_{t + 1}},{\bar y_{t + 1}} = {y_{t + 1}}
\end{equation}

\subsection{Primal Dual Gap }
\newtheorem{prop}{Proposition}[section]
\begin{prop}
With convex concave function defined in (\ref{eq1.1}), Primal Dual Gap ${E_{t + 1}}$  with respect to aggregation point $({\bar x_{t + 1}},{\bar y_{t + 1}})$  is defined as \\
\vspace{-5mm}
\begin{equation}
\begin{array}{l}
{E_{t + 1}} = L({{\bar x}_{t + 1}},y) - L(x,{{\bar y}_{t + 1}})\\
= f({{\bar x}_{t + 1}}) + \left\langle {A{{\bar x}_{t + 1}},y} \right\rangle  - g(y) - \left[ {f(x) + \left\langle {Ax,{{\bar y}_{t + 1}}} \right\rangle  - g({{\bar y}_{t + 1}})} \right]\\
{\rm{ = }}f({{\bar x}_{t + 1}}) - f(x){\rm{ + }}\left\langle {A{{\bar x}_{t + 1}},y} \right\rangle  - \left\langle {Ax,{{\bar y}_{t + 1}}} \right\rangle  - \left[ {g(y) - g({{\bar y}_{t + 1}})} \right]
\end{array} \label{eq2.8}
\end{equation}
For ${\theta _t} \in [0,1]$, it holds that
\begin{equation}
\begin{array}{l}
\left[ {{E_{t + 1}} - (1 - {\theta _t}){E_t}} \right]/{\theta _t} \le \left\{ {f({{\bar x}_{t + 1}}) - f(x) - (1 - {\theta _t})\left[ {f({{\bar x}_t}) - f(x)} \right]} \right\}/{\theta _t}\\
+ \left\langle {A{x_{t + 1}},y} \right\rangle  - \left\langle {Ax,{y_{t + 1}}} \right\rangle  + g({y_{t + 1}}) - g(y)
\end{array} \label{eq2.9}
\end{equation}
\end{prop}
\begin{proof}
Using definition in (\ref{eq2.8}), we have
\begin{equation}	
\begin{array}{l}
{E_{t + 1}} - (1 - {\theta _t}){E_t}\\
= f({{\bar x}_{t + 1}}) - f(x){\rm{ + }}\left\langle {A{{\bar x}_{t + 1}},y} \right\rangle  - \left\langle {Ax,{{\bar y}_{t + 1}}} \right\rangle  - \left[ {g(y) - g({{\bar y}_{t + 1}})} \right]\\
- (1 - {\theta _t})\left[ {f({{\bar x}_t}) - f(x){\rm{ + }}\left\langle {A{{\bar x}_t},y} \right\rangle  - \left\langle {Ax,{{\bar y}_t}} \right\rangle  - \left[ {g(y) - g({{\bar y}_t})} \right]} \right]\\
= f({{\bar x}_{t + 1}}) - f(x) - (1 - {\theta _t})\left[ {f({{\bar x}_t}) - f(x)} \right]\\
+ \left\langle {A{{\bar x}_{t + 1}},y} \right\rangle  - \left\langle {Ax,{{\bar y}_{t + 1}}} \right\rangle  - (1 - {\theta _t})[\left\langle {A{{\bar x}_t},y} \right\rangle  - \left\langle {Ax,{{\bar y}_t}} \right\rangle ]\\
- \left[ {g(y) - g({{\bar y}_{t + 1}})} \right] - (1 - {\theta _t})\left[ { - \left[ {g(y) - g({{\bar y}_t})} \right]} \right]
\end{array}\label{eq2.10}
\end{equation}
Using (\ref{eq2.3}) and (\ref{eq2.6}), we have
\begin{equation}
\begin{array}{l}
\left\langle {A{{\bar x}_{t + 1}},y} \right\rangle  - \left\langle {Ax,{{\bar y}_{t + 1}}} \right\rangle  - (1 - {\theta _t})\left( {\left\langle {A{{\bar x}_t},y} \right\rangle  - \left\langle {Ax,{{\bar y}_t}} \right\rangle } \right)\\
= {\theta _t}\left[ {\left\langle {A{x_{t + 1}},y} \right\rangle  - \left\langle {Ax,{y_{t + 1}}} \right\rangle } \right]
\end{array}\label{eq2.11}
\end{equation}
\begin{equation}
\begin{array}{l}
- \left[ {g(y) - g({{\bar y}_{t + 1}})} \right] - (1 - {\theta _t})\left( { - \left[ {g(y) - g({{\bar y}_t})} \right]} \right)\\
=  - {\theta _t}g(y) + g({{\bar y}_{t + 1}}) - (1 - {\theta _t})g({{\bar y}_t})\\
=  - {\theta _t}g(y) + g((1 - {\theta _t}){{\bar y}_t} + {\theta _t}{y_{t + 1}}) - (1 - {\theta _t})g({{\bar y}_t})\\
\le  - {\theta _t}g(y) + (1 - {\theta _t})g({{\bar y}_t}) + {\theta _t}g({y_{t + 1}}) - (1 - {\theta _t})g({{\bar y}_t})\\
= {\theta _t}(g({y_{t + 1}}) - g(y))
\end{array}  \label{eq2.12}
\end{equation}
where the second equality follows from (\ref{eq2.6}) and the inequality from convexity of $g(y)$.\\
Substituting (\ref{eq2.11}) and (\ref{eq2.12}) into (\ref{eq2.10}), we have
\begin{equation}
\begin{array}{l}
{E_{t + 1}} - (1 - {\theta _t}){E_t} \le f({{\bar x}_{t + 1}}) - f(x) - (1 - {\theta _t})\left[ {f({{\bar x}_t}) - f(x)} \right]\\
+ {\theta _t}\left[ {\left\langle {A{x_{t + 1}},y} \right\rangle  - \left\langle {Ax,{y_{t + 1}}} \right\rangle } \right] + {\theta _t}(g({y_{t + 1}}) - g(y))
\end{array} \label{eq2.13}
\end{equation}
Dividing both sides of above inequality by ${\theta _t}$, we have (\ref{eq2.9}).

\end{proof}

\begin{prop}
 Setting ${\theta _t} = 2/(t + 1)$ in LDPD, it holds that
\begin{equation}
{\bar x_{k + 1}} = \sum\limits_{t = 1}^k {(t \cdot {x_{t + 1}})} /\sum\limits_{t = 1}^k t \begin{array}{*{20}{c}}
,&{}
\end{array}{\bar y_{k + 1}} = \sum\limits_{t = 1}^k {(t \cdot {y_{t + 1}})} /\sum\limits_{t = 1}^k t  \label{eq2.14}
\end{equation}

\begin{proof}
The result can be found in \cite{Chen2013Optimal}, below is the simple proof.  In view of the similarity of (\ref{eq2.3}) and (\ref{eq2.6}), we only need to prove the above left equality.

For $k=1$ , we need to show 
${\bar x_2} = \sum\limits_{t = 1}^1 {(t \cdot {x_{t + 1}})} /\sum\limits_{t = 1}^1 t  = {x_2}$ , which is verified by plugging  $t=1$ into  
${\bar x_{t + 1}} = (1 - {\theta _t}){\bar x_t} + {\theta _t}{x_{t + 1}}$
(\ref{eq2.3}) with ${\theta _1} = 1$. Suppose (\ref{eq2.14}) holds for $k=N$, we need to confirm it holds for $k=N+1$, which demands
\begin{equation}
{\bar x_{N + 2}} = \sum\limits_{t = 1}^{N + 1} {(t \cdot {x_{t + 1}})} /\sum\limits_{t = 1}^{N + 1} t  = \frac{{1{x_2} + 2{x_3} + N{x_{N + 1}} + (N + 1){x_{N + 2}}}}{{(N + 1)(N + 2)/2}}\label{eq2.15}
\end{equation}
We already have
${\bar x_{N + 1}} = \sum\limits_{t = 1}^N {(t \cdot {x_{t + 1}})} /\sum\limits_{t = 1}^N t $ , using (\ref{eq2.3}), we have
\[\begin{array}{l}
{{\bar x}_{N + 2}} = (1 - {\theta _{N + 1}}){{\bar x}_{N + 1}} + {\theta _{N + 1}}{x_{N + 2}} = \frac{N}{{N + 2}}{{\bar x}_{N + 1}} + \frac{2}{{N + 2}}{x_{N + 2}}\\
= \frac{N}{{N + 2}}\left[ {\frac{{1{x_2} + 2{x_3} + .. + N{x_{N + 1}}}}{{1 + 2 + .. + N}}} \right] + \frac{2}{{N + 2}}{x_{N + 2}} = \frac{{1{x_2} + 2{x_3} + .. + N{x_{N + 1}} + (N + 1){x_{N + 2}}}}{{(N + 2)(N + 1)/2}}
\end{array}\]
which proves (\ref{eq2.15}), hence finish proof of (\ref{eq2.14}).

\end{proof}

\end{prop}

\newtheorem{lema}{Lemma}[section]
	
\section{Convergence Analysis for LDPD}
In this section, let us firstly establish some preliminary results, then develop convergence analysis for partial or full acceleration, which corresponds to weakly or strongly convex problem.

\subsection{ Preliminary Results}
\begin{lema}
	Let ${\mu _g} \ge 0$  denote strongly convexity parameter of function $g(y)$ . For iteration produced by LDPD, we have
	\begin{equation}
	\begin{array}{l}
	g({y_{t + 1}}) - g(y) \le \left\langle {y - {y_{t + 1}},{y_{t + 1}} - {y_t}} \right\rangle /{\tau _t}\\
	- {\mu _g}||y - {y_{t + 1}}|{|^2}/2 - \left\langle {A{x_{t + 1}},y - {y_{t + 1}}} \right\rangle 
	\end{array} \label{eq3.1}
	\end{equation}
\end{lema}

\begin{proof}
Using sub-differential operator, optimal condition of (\ref{eq2.4}) can be written as 	
\begin{equation}
0 \in (1/{\tau _t})({y_{t + 1}} - {y_t} - {\tau _t}A{x_{t + 1}}) + \partial g({y_{t + 1}})
\label{eq3.2}
\end{equation}
where $\partial g({y_{t + 1}})$ denotes a subgradient of $g(y)$ at ${y_{t + 1}}$.
Strongly convexity of $g(y)$  implies
\begin{equation}
g(y) - g({y_{t + 1}}) \ge \left\langle {y - {y_{t + 1}},\partial g({y_{t + 1}})} \right\rangle  + {\mu _g}||y - {y_{t + 1}}|{|^2}/2
\label{eq3.3}
\end{equation}
Substituting (\ref{eq3.2}) into (\ref{eq3.3}), we have
\begin{equation}
g(y) - g({y_{t + 1}}) \ge \left\langle {{y_{t + 1}} - y,({y_{t + 1}} - {y_t} - {\tau _t}A{x_{t + 1}})/{\tau _t}} \right\rangle  + {\mu _g}||y - {y_{t + 1}}|{|^2}/2
\end{equation}
Which leads to (\ref{eq3.1}).
\end{proof}

\begin{lema}
	Suppose function $f(x)$  is strongly convex with strongly convexity parameter 
	${\mu _f} \ge 0$  and has continuous gradient with lipschitz constant ${L_f}$ , for iteration in LDPD, we have
	\begin{equation}
	\begin{array}{l}
	f({{\bar x}_{t + 1}}) - f(x) - (1 - {\theta _t})\left[ {f({{\bar x}_t}) - f(x)} \right]\\
	\le {\theta _t}\left\langle {\nabla f({{\hat x}_t}),{x_{t + 1}} - x} \right\rangle  + {L_f}\theta _t^2||{x_{t + 1}} - {x_t}|{|^2}/2 - {\mu _f}\theta _t^2||x - {x_t}|{|^2}/2
	\end{array}
	\label{eq3.5}
   \end{equation}
\end{lema}

  \begin{proof}
  	Since $f(z)$  is strongly convex with parameter ${\mu _f} \ge 0$  , we have
  	\begin{equation}
    f({\hat x_t}) - f(z) \le \left\langle {\nabla f({{\hat x}_t}),{{\hat x}_t} - z} \right\rangle  - {\mu _f}||z - {\hat x_t}|{|^2}/2
    \label{eq3.6}
  	\end{equation}

   The Lipschitz continuousness of gradient of $f(z)$  implies
   \begin{equation}
   f({\bar x_{t + 1}}) - f({\hat x_t}) \le \left\langle {\nabla f({{\hat x}_t}),{{\bar x}_{t + 1}} - {{\hat x}_t}} \right\rangle  + {L_f}||{\bar x_{t + 1}} - {\hat x_t}|{|^2}/2
   \label{eq3.7}
   \end{equation}
   Summing up (\ref{eq3.6}) and (\ref{eq3.7}) leads to
   \begin{equation}
   f({\bar x_{t + 1}}) \le f(z) + \left\langle {\nabla f({{\hat x}_t}),{{\bar x}_{t + 1}} - z} \right\rangle  + {L_f}||{\bar x_{t + 1}} - {\hat x_t}|{|^2}/2 - {\mu _f}||z - {\hat x_t}|{|^2}/2
   \end{equation}
   Since above inequality holds for any $z$ , letting $z = (1 - {\theta _t}){\bar x_t} + {\theta _t}x$  , we have
   \[\begin{array}{l}
   f({{\bar x}_{t + 1}}) \le f((1 - {\theta _t}){{\bar x}_t} + {\theta _t}x) + \left\langle {\nabla f({{\hat x}_t}),{{\bar x}_{t + 1}} - [(1 - {\theta _t}){{\bar x}_t} + {\theta _t}x]} \right\rangle \\
   + {L_f}||{{\bar x}_{t + 1}} - {{\hat x}_t}|{|^2}/2 - {\mu _f}||(1 - {\theta _t}){{\bar x}_t} + {\theta _t}x - {{\hat x}_t}|{|^2}/2\\
   \le (1 - {\theta _t})f({{\bar x}_t}) + {\theta _t}f(x) + \left\langle {\nabla f({{\hat x}_t}),{\theta _t}{x_{t + 1}} - {\theta _t}x} \right\rangle  + {L_f}||{{\bar x}_{t + 1}} - {{\hat x}_t}|{|^2}/2 - {\mu _f}||{\theta _t}x - {\theta _t}{x_t}|{|^2}/2
   \end{array}\]
   Where the inequality comes from convexity of $f()$ and (\ref{eq2.1}),(\ref{eq2.3}), the above inequality leads to (\ref{eq3.5}) immediately with
   ${\bar x_{t + 1}} - {\hat x_t} = {\theta _t}({x_{t + 1}} - {x_t})$.
  \end{proof}

\begin{prop}
 For iterations in LDPD, we have
 \begin{equation}
 \begin{array}{l}
 \left\{ {f({{\bar x}_{t + 1}}) - f(x) - (1 - {\theta _t})\left[ {f({{\bar x}_t}) - f(x)} \right]} \right\}/{\theta _t} \le \left\langle {x - {x_{t + 1}},{x_{t + 1}} - {x_t}} \right\rangle /{\eta _t}\\
 + {L_f}\theta _t^{}||{x_{t + 1}} - {x_t}|{|^2}/2 - {\mu _f}\theta _t^{}||x - {x_t}|{|^2}/2 + \left\langle {Ax - A{x_{t + 1}},{{\hat y}_t}} \right\rangle 
  \end{array}
 \label{eq3.9}
 \end{equation}	
\vspace{-5mm}
 \begin{proof}
Using variation inequality, the optimal condition of (\ref{eq2.2}) can be written as
\[\left\langle {x - {x_{t + 1}},{x_{t + 1}} - {x_t} + {\eta _t}(\nabla f({{\hat x}_t}) + {A^T}{{\hat y}_t})} \right\rangle  \ge 0\]
which means
\[\left\langle {x - {x_{t + 1}},{x_{t + 1}} - {x_t} + {\eta _t}{A^T}{{\hat y}_t}} \right\rangle /{\eta _t} \ge \left\langle {{x_{t + 1}} - x,\nabla f({{\hat x}_t})} \right\rangle \]
Substituting above inequality into (\ref{eq3.5}) leads to 
\[\begin{array}{l}
f({{\bar x}_{t + 1}}) - f(x) - (1 - {\theta _t})\left[ {f({{\bar x}_t}) - f(x)} \right] \le {L_f}\theta _t^2||{x_{t + 1}} - {x_t}|{|^2}/2 - {\mu _f}\theta _t^2||x - {x_t}|{|^2}/2\\
+ \left\langle {x - {x_{t + 1}},{x_{t + 1}} - {x_t}} \right\rangle {\theta _t}/{\eta _t} + {\theta _t}\left\langle {x - {x_{t + 1}},{A^T}{{\hat y}_t}} \right\rangle 
\end{array}\]
Dividing Both sides of above inequality by ${\theta _t}$ , we have (\ref{eq3.9}).
\end{proof}	
\end{prop}

\begin{prop}It holds for iteration in LDPD that
\begin{equation}
\begin{array}{l}
\left[ {{E_{t + 1}} - (1 - {\theta _t}){E_t}} \right]/{\theta _t} \le \left\langle {x - {x_{t + 1}},{x_{t + 1}} - {x_t}} \right\rangle /{\eta _t} + {L_f}\theta _t^{}||{x_{t + 1}} - {x_t}|{|^2}/2\\
- {\mu _f}\theta _t^{}||x - {x_t}|{|^2}/2 + \left\langle {y - {y_{t + 1}},{y_{t + 1}} - {y_t}} \right\rangle /{\tau _t} - {\mu _g}||y - {y_{t + 1}}|{|^2}/2\\
+ \left\langle {Ax - A{x_{t + 1}},{y_t} - {y_{t + 1}}} \right\rangle  - \left\langle {Ax - A{x_t},{y_{t - 1}} - {y_t}} \right\rangle {\alpha _t} + \left\langle {A{x_t} - A{x_{t + 1}},{y_t} - {y_{t - 1}}} \right\rangle {\alpha _t}
\end{array}
\label{eq3.10}
\end{equation}	
\end{prop}
\begin{proof}
 Substituting (\ref{eq3.9}) and (\ref{eq3.1}) into (\ref{eq2.9}), we have
 \begin{equation}
\begin{array}{l}
 \left[ {{E_{t + 1}} - (1 - {\theta _t}){E_t}} \right]/{\theta _t} \le \left\langle {x - {x_{t + 1}},{x_{t + 1}} - {x_t}} \right\rangle /{\eta _t}\\
 + {L_f}\theta _t^{}||{x_{t + 1}} - {x_t}|{|^2}/2 - {\mu _f}\theta _t^{}||x - {x_t}|{|^2}/2 + \left\langle {Ax - A{x_{t + 1}},{{\hat y}_t}} \right\rangle \\
 + \left\langle {A{x_{t + 1}},y} \right\rangle  - \left\langle {Ax,{y_{t + 1}}} \right\rangle  + \left\langle {y - {y_{t + 1}},{y_{t + 1}} - {y_t}} \right\rangle /{\tau _t}\\
 - {\mu _g}||y - {y_{t + 1}}|{|^2}/2 - \left\langle {A{x_{t + 1}},y - {y_{t + 1}}} \right\rangle 
 \end{array}
 \label{eq3.11}
 \end{equation}
 Noting the bilinear inner product term in above inequality, we see
\begin{equation}
\begin{array}{l}
\left\langle {Ax - A{x_{t + 1}},{{\hat y}_t}} \right\rangle  + \left\langle {A{x_{t + 1}},y} \right\rangle  - \left\langle {Ax,{y_{t + 1}}} \right\rangle  - \left\langle {A{x_{t + 1}},y - {y_{t + 1}}} \right\rangle \\
= \left\langle {Ax - A{x_{t + 1}},{{\hat y}_t}} \right\rangle  - \left\langle {Ax,{y_{t + 1}}} \right\rangle  + \left\langle {A{x_{t + 1}},{y_{t + 1}}} \right\rangle \\
= \left\langle {Ax - A{x_{t + 1}},{y_t} - {y_{t + 1}}} \right\rangle  + \left\langle {Ax - A{x_{t + 1}},({y_t} - {y_{t - 1}}){\alpha _t}} \right\rangle \\
= \left\langle {Ax - A{x_{t + 1}},{y_t} - {y_{t + 1}}} \right\rangle  - \left\langle {Ax - A{x_t},{y_{t - 1}} - {y_t}} \right\rangle {\alpha _t}\\
+ \left\langle {A{x_t} - A{x_{t + 1}},{y_t} - {y_{t - 1}}} \right\rangle {\alpha _t}
\end{array}
\label{eq3.12}
\end{equation}
where the above second equality comes from 
${\hat y_t} = ({y_t} - {y_{t - 1}}){\alpha _t} + {y_t}$  by (\ref{eq2.5}).
Substituting (\ref{eq3.12}) into (\ref{eq3.11}), we have (\ref{eq3.10}).
 
\end{proof}

\begin{lema}
Setting parameters of LDPD as
\begin{equation}
{\theta _t} = 2/(t + 1),\begin{array}{*{20}{c}}
{}
\end{array}{\alpha _t} = (t - 1)/t,\begin{array}{*{20}{c}}
{}
\end{array}(t + 1){\tau _{t + 1}} \ge t{\tau _t},\begin{array}{*{20}{c}}
{}
\end{array}1/{\eta _t} \ge 2{L_f}/(t + 1) + ||A|{|^2}{\tau _t}
\label{eq3.13}
\end{equation}
then we have
\begin{equation}
\begin{array}{l}
t(t + 1){E_{t + 1}} - t(t - 1){E_t} \le \left( {||x - {x_t}|{|^2} - ||x - {x_{t + 1}}|{|^2}} \right)t/{\eta _t} - t(t + 1)\frac{{{\mu _f}}}{2}||x - {x_t}|{|^2}\\
- t{\mu _g}||y - {y_{t + 1}}|{|^2} + \left( {||y - {y_t}|{|^2} - ||y - {y_{t + 1}}|{|^2} - ||{y_{t + 1}} - {y_t}|{|^2}} \right)\frac{t}{{{\tau _t}}}\\
+ 2t\left\langle {Ax - A{x_{t + 1}},{y_t} - {y_{t + 1}}} \right\rangle  - 2(t - 1)\left\langle {Ax - A{x_t},{y_{t - 1}} - {y_t}} \right\rangle  + \frac{{(t - 1)}}{{{\tau _{t - 1}}}}||{y_t} - {y_{t - 1}}|{|^2}
\end{array}
\label{eq3.14}
\end{equation}
\end{lema}
\begin{proof}
Using ${\theta _t} = 2/(t + 1)$,${\alpha _t} = (t - 1)/t$ into (\ref{eq3.10}) and multiply both sides with $t$  yields
\begin{equation}
\begin{array}{l}
t{E_{t + 1}}\frac{{t + 1}}{2} - t{E_t}\frac{{t - 1}}{2} \le t\left\langle {x - {x_{t + 1}},{x_{t + 1}} - {x_t}} \right\rangle /{\eta _t}\\
+ {L_f}t\frac{2}{{t + 1}}||{x_{t + 1}} - {x_t}|{|^2}/2 - t{\mu _f}\frac{{t + 1}}{2}||x - {x_t}|{|^2}/2\\
+ t\left\langle {y - {y_{t + 1}},{y_{t + 1}} - {y_t}} \right\rangle /{\tau _t} - t{\mu _g}||y - {y_{t + 1}}|{|^2}/2\\
+ t\left\langle {Ax - A{x_{t + 1}},{y_t} - {y_{t + 1}}} \right\rangle  - \left\langle {Ax - A{x_t},{y_{t - 1}} - {y_t}} \right\rangle t{\alpha _t}\\
+ \left\langle {A{x_t} - A{x_{t + 1}},{y_t} - {y_{t - 1}}} \right\rangle t{\alpha _t}
\end{array}
\label{eq3.15}
\end{equation}
Using $t{\alpha _t} = (t - 1)$, we have	
\begin{equation}
\begin{array}{l}
\left\langle {A{x_t} - A{x_{t + 1}},{y_t} - {y_{t - 1}}} \right\rangle t{\alpha _t} = \left\langle {A{x_t} - A{x_{t + 1}},{y_t} - {y_{t - 1}}} \right\rangle (t - 1)\\
\le ||A{x_t} - A{x_{t + 1}}||||{y_t} - {y_{t - 1}}||(t - 1)\\
\le \left( {||A{x_t} - A{x_{t + 1}}|{|^2}{\tau _{t - 1}} + ||{y_t} - {y_{t - 1}}|{|^2}/{\tau _{t - 1}}} \right)(t - 1)/2\\
\le ||A|{|^2}||{x_t} - {x_{t + 1}}|{|^2}t{\tau _t}/2 + (t - 1)||{y_t} - {y_{t - 1}}|{|^2}/(2{\tau _{t - 1}})
\end{array}
\label{eq3.16}
\end{equation}

\noindent the above last inequality uses $(t - 1){\tau _{t - 1}} \le t{\tau _t}$  in (\ref{eq3.13}). Plugging (\ref{eq3.16}) into (\ref{eq3.15}) produces 
\[\begin{array}{l}
t{E_{t + 1}}\frac{{t + 1}}{2} - t{E_t}\frac{{t - 1}}{2} \le \left( {||x - {x_t}|{|^2} - ||x - {x_{t + 1}}|{|^2}} \right)\frac{t}{{2{\eta _t}}}\\
- \frac{t}{2}||{x_{t + 1}} - {x_t}|{|^2}(1/{\eta _t} - {L_f}\frac{2}{{t + 1}} - ||A|{|^2}{\tau _t}) - t{\mu _f}\frac{{t + 1}}{2}||x - {x_t}|{|^2}/2\\
- t{\mu _g}||y - {y_{t + 1}}|{|^2}/2 + \left( {||y - {y_t}|{|^2} - ||y - {y_{t + 1}}|{|^2} - ||{y_{t + 1}} - {y_t}|{|^2}} \right)t/(2{\tau _t})\\
+ t\left\langle {Ax - A{x_{t + 1}},{y_t} - {y_{t + 1}}} \right\rangle  - (t - 1)\left\langle {Ax - A{x_t},{y_{t - 1}} - {y_t}} \right\rangle  + (t - 1)||{y_t} - {y_{t - 1}}|{|^2}/(2{\tau _{t - 1}})
\end{array}\]
Using ${\eta _t} \ge {L_f}(t + 1)/2 + ||A|{|^2}{\tau _t}$ in (\ref{eq3.13}) leads to (\ref{eq3.14}).

\end{proof}

\subsection{ Partial Acceleration by Smooth Primal Component}\label{section3.2}
Here let us recover the state-of-the-art result with respect to partial acceleration \cite{Chen2013Optimal} by using smooth property of primal component for weakly convex problem. 
\begin{theorem}
 
For saddle point problem (\ref{eq1.1}), if both primal component $f(x)$ and dual component $g(y)$ are merely weakly convex, namely 
 ${\mu _f} = {\mu _g} = 0$ , letting $N$ denoting maximum iteration number and
 \begin{equation}
 {\theta _t} = \frac{2}{{t + 1}},\begin{array}{*{20}{c}}
 {}
 \end{array}{\alpha _t} = \frac{{t - 1}}{t},\begin{array}{*{20}{c}}
 {}
 \end{array}{\tau _t} = \frac{t}{N},\begin{array}{*{20}{c}}
 {}
 \end{array}{\eta _t} = \frac{t}{{2{L_f} + N||A|{|^2}}}
 \label{eq3.17}
 \end{equation}
 then it holds for LDPD that
 \begin{equation}
 {E_{N + 1}} \le \frac{{2{L_f}}}{{N(N + 1)}}||x - {x_1}|{|^2} + \frac{1}{{N + 1}}\left( {||A|{|^2}||x - {x_1}|{|^2} + ||y - {y_1}|{|^2}} \right)
 \label{eq3.18}
 \end{equation}
 where ${L_f}$   denotes gradient lipschitz constant of smooth primal component $f(x)$ \\and 
 ${E_{N + 1}} = L({\bar x_{N + 1}},y) - L(x,{\bar y_{N + 1}})$ with aggregation point  
 $({\bar x_{N + 1}},{\bar y_{N + 1}})$ defined as (\ref{eq2.14}).
\end{theorem}
\begin{proof}
Note that $t \le N$
  and the primal step size   ${\eta _t}$  satisfy (\ref{eq3.13}), i.e.,
\[1/{\eta _t} = (2{L_f} + N||A|{|^2})/t \ge {L_f}2/(t + 1) + ||A|{|^2}t/N = {L_f}2/(t + 1) + ||A|{|^2}{\tau _t}\]
$(t + 1){\tau _{t + 1}} = {(t + 1)^2}/N > {t^2}/N = t{\tau _t}$
we can apply (\ref{eq3.14}) . Plugging
${\mu _f} = {\mu _g} = 0$   and (\ref{eq3.13}) into (\ref{eq3.14}), we have
\[\begin{array}{l}
t(t + 1){E_{t + 1}} - t(t - 1){E_t} \le \left( {||x - {x_t}|{|^2} - ||x - {x_{t + 1}}|{|^2}} \right)(2{L_f} + N||A|{|^2})\\
+ \left( {||y - {y_t}|{|^2} - ||y - {y_{t + 1}}|{|^2} - ||{y_{t + 1}} - {y_t}|{|^2}} \right)N\\
+ 2t\left\langle {Ax - A{x_{t + 1}},{y_t} - {y_{t + 1}}} \right\rangle  - 2(t - 1)\left\langle {Ax - A{x_t},{y_{t - 1}} - {y_t}} \right\rangle  + N||{y_t} - {y_{t - 1}}|{|^2}
\end{array}\]
Summing up above inequality over $t=1,2,...,N$ leads to
\begin{equation}
\begin{array}{l}
N(N + 1){E_{N + 1}} \le (2{L_f} + N||A|{|^2})(||x - {x_1}|{|^2} - ||x - {x_{N + 1}}|{|^2}) + N(||y - {y_1}|{|^2})\\
- \sum\limits_{t = 1}^N {N||{y_t} - {y_{t + 1}}{\rm{|}}{{\rm{|}}^2}}  + 2N\left\langle {Ax - A{x_{N + 1}},{y_N} - {y_{N + 1}}} \right\rangle  + \sum\limits_{t = 1}^N {N||{y_t} - {y_{t - 1}}|{|^2}} 
\end{array}
\label{eq3.19}
\end{equation}
Noting that
\[\begin{array}{l}
2\left\langle {Ax - A{x_{N + 1}},{y_N} - {y_{N + 1}}} \right\rangle  \le 2||Ax - A{x_{N + 1}}|| \cdot ||{y_N} - {y_{N + 1}}||\\
\le \left( {||A|{|^2}||x - {x_{N + 1}}|{|^2} + ||{y_N} - {y_{N + 1}}|{|^2}} \right)
\end{array}\]
Substituting above inequality into (3.19), we have
 \[{E_{N + 1}}N(N + 1) \le (2{L_f} + N||A|{|^2})||x - {x_1}|{|^2} + N(||y - {y_1}|{|^2})\]
which yields (\ref{eq3.18}) immediately.
\end{proof}

\subsection{Full Acceleration by Smooth Primal and Strongly Convex Dual Component}\label{section3.3}
\begin{theorem}\label{theorem3.2}
For saddle point problem (\ref{eq1.1}), if  primal component $f(x)$ is weakly convex with ${\mu _f} = 0$ , dual component $g(y)$ 	is strongly convex with 
${\mu _g} > 0$, setting parameter as
\begin{equation}
{\theta _t} = \frac{2}{{t + 1}},\begin{array}{*{20}{c}}
{}
\end{array}{\alpha _t} = \frac{{t - 1}}{t},\begin{array}{*{20}{c}}
{}
\end{array}\tau  = \frac{3}{{{\mu _g}}},\begin{array}{*{20}{c}}
{}
\end{array}{\tau _t} = \frac{\tau }{t},\begin{array}{*{20}{c}}
{}
\end{array}{\eta _t} = \frac{t}{{2{L_f} + \tau ||A|{|^2}}}
\label{eq3.20}
\end{equation}	
then it holds for LDPD that
\begin{equation}
{E_{k + 1}} \le \frac{{(2{L_f} + \tau ||A|{|^2})||x - {x_1}|{|^2}}}{{k(k + 1)}} + \frac{{||y - {y_1}|{|^2}}}{{k(k + 1)\tau }}
\label{eq3.21}
\end{equation}
with 
${E_{k + 1}} = L({\bar x_{k + 1}},y) - L(x,{\bar y_{k + 1}})$  and aggregation point $({\bar x_{k + 1}},{\bar y_{k + 1}})$  is defined in (\ref{eq2.14}).
\end{theorem}
\begin{proof}
Since $(t + 1){\tau _{t + 1}} = t{\tau _t} = \tau$,
$1/{\eta _t} = 2{L_f}/t + ||A|{|^2}\tau /t \ge 2{L_f}/(t + 1) + ||A|{|^2}{\tau _t}$, condition (\ref{eq3.13}) is satisfied, one can apply (\ref{eq3.14}). Plugging ${\mu _f} = 0$  and (\ref{eq3.20}) into (\ref{eq3.14}) produces

\[\begin{array}{l}
t(t + 1){E_{t + 1}} - t(t - 1){E_t} \le \left( {||x - {x_t}|{|^2} - ||x - {x_{t + 1}}|{|^2}} \right)(2{L_f} + \tau ||A|{|^2})\\
- t{\mu _g}||y - {y_{t + 1}}|{|^2} + \left( {||y - {y_t}|{|^2} - ||y - {y_{t + 1}}|{|^2} - ||{y_{t + 1}} - {y_t}|{|^2}} \right){t^2}/\tau \\
+ 2t\left\langle {Ax - A{x_{t + 1}},{y_t} - {y_{t + 1}}} \right\rangle  - 2(t - 1)\left\langle {Ax - A{x_t},{y_{t - 1}} - {y_t}} \right\rangle  + ||{y_t} - {y_{t - 1}}|{|^2}{(t - 1)^2}/\tau 
\end{array}\]
Summing up above inequality through $t=1,2,...k$  delivers
\begin{equation}
\begin{array}{l}
k(k + 1){E_{k + 1}} \le (2{L_f} + \tau ||A|{|^2})\left( {||x - {x_1}|{|^2} - ||x - {x_{k + 1}}|{|^2}} \right) - ||{y_{k + 1}} - {y_k}{\rm{|}}{{\rm{|}}^2}{k^2}/\tau \\
+ \sum\limits_{t = 1}^k {\left[ {||y - {y_t}|{|^2}{t^2}/\tau  - ({t^2}/\tau  + t{\mu _g})||y - {y_{t + 1}}|{|^2}} \right]}  + 2k\left\langle {Ax - A{x_{k + 1}},{y_k} - {y_{k + 1}}} \right\rangle 
\end{array}
\label{eq3.22}
\end{equation}
Noting that

\[\begin{array}{l}
2k\left\langle {Ax - A{x_{k + 1}},{y_k} - {y_{k + 1}}} \right\rangle  \le 2k||{y_k} - {y_{k + 1}}||(||A|| \cdot ||x - {x_{k + 1}}||)\\
\le {k^2}||{y_k} - {y_{k + 1}}|{|^2}/\tau  + \tau ||A|{|^2}||x - {x_{k + 1}}|{|^2}
\end{array}\]
Substituting above inequality into (\ref{eq3.22}) yields
\begin{equation}
\begin{array}{l}
k(k + 1){E_{k + 1}} \le (2{L_f} + \tau ||A|{|^2})||x - {x_1}|{|^2}\\
+ \sum\limits_{t = 1}^k {\left[ {||y - {y_t}|{|^2}{t^2}/\tau  - ({t^2}/\tau  + t{\mu _g})||y - {y_{t + 1}}|{|^2}} \right]} 
\end{array}
\label{eq3.23}
\end{equation}
Letting
\begin{equation}
{t^2}/\tau  + t{\mu _g} \ge {(t + 1)^2}/\tau 
\label{eq3.24}
\end{equation}
Plugging (\ref{eq3.24}) into (\ref{eq3.23}) gives
\[k(k + 1){E_{k + 1}} \le (2{L_f} + \tau ||A|{|^2})||x - {x_1}|{|^2} + ||y - {y_1}|{|^2}/\tau \]
which leads to (\ref{eq3.21}) instantly. Inequality (\ref{eq3.24}) demands ${t^2} + \tau t{\mu _g} \ge {(t + 1)^2}$, that is 
$\tau t{\mu _g} \ge 2t + 1$,which is satisfied by $\tau  = 3/{\mu _g}$ in (\ref{eq3.20}) since $(3/{\mu _g})t{\mu _g} \ge 2t + 1$.

\end{proof}

\begin{lema}
For convex function $f(x)$  and $g(y)$  , let ${\mu _f} \ge 0$, and
${\mu _g} \ge 0$  denote their strongly convexity parameter respectively, and let $({x^*},{y^*})$  denote a saddle point of (\ref{eq1.1}). Let	
\begin{equation}
E_{t + 1}^* = L({\bar x_{t + 1}},{y^*}) - L({x^*},{\bar y_{t + 1}})
\label{eq3.25}
\end{equation}
we have
\begin{equation}
E_{t + 1}^* \ge {\mu _f}||{\bar x_{t + 1}} - {x^*}|{|^2}/2 + {\mu _g}||{\bar y_{t + 1}} - {y^*}|{|^2}/2
\label{eq3.26}
\end{equation}
\end{lema}
\begin{proof}
KKT condition in (\ref{eq1.1}) can be written as
\begin{equation}
 - {A^T}{y^*} \in \partial f({x^*}),\begin{array}{*{20}{c}}
{}&{}
\end{array}A{x^*} \in \partial g({y^*})
\label{eq3.27}
\end{equation}
Using strongly convexity of $f(x)$, $g(y)$ and noting (\ref{eq3.27}), we have
\begin{equation}
f(x) - f({x^*}) - \left\langle {x - {x^*}, - {A^T}{y^*}} \right\rangle  \ge {\mu _f}||x - {x^*}|{|^2}/2
\label{eq3.28}
\end{equation}
\begin{equation}
g(y) - g({y^*}) - \left\langle {A{x^*},y - {y^*}} \right\rangle  \ge {\mu _g}||y - {y^*}|{|^2}/2
\label{eq3.29}
\end{equation}
Using $L(x,y)$ in (1.1) , we have
\[\begin{array}{l}
E_{t + 1}^* = L({{\bar x}_{t + 1}},{y^*}) - L({x^*},{{\bar y}_{t + 1}})\\
= f({{\bar x}_{t + 1}}) + \left\langle {A{{\bar x}_{t + 1}},{y^*}} \right\rangle  - g({y^*}) - \left[ {f({x^*}) + \left\langle {A{x^*},{{\bar y}_{t + 1}}} \right\rangle  - g({{\bar y}_{t + 1}})} \right]\\
{\rm{ = }}f({{\bar x}_{t + 1}}) - f({x^*}) - \left\langle {{{\bar x}_{t + 1}}, - {A^T}{y^*}} \right\rangle  + g({{\bar y}_{t + 1}}) - g({y^*}) - \left\langle {A{x^*},{{\bar y}_{t + 1}}} \right\rangle \\
= f({{\bar x}_{t + 1}}) - f({x^*}) - \left\langle {{{\bar x}_{t + 1}} - {x^*}, - {A^T}{y^*}} \right\rangle  + g({{\bar y}_{t + 1}}) - g({y^*}) - \left\langle {A{x^*},{{\bar y}_{t + 1}} - {y^*}} \right\rangle \\
\ge {\mu _f}||{{\bar x}_{t + 1}} - {x^*}|{|^2}/2 + {\mu _g}||{{\bar y}_{t + 1}} - {y^*}|{|^2}/2
\end{array}\]
Where the last inequality comes from (\ref{eq3.28}) and (\ref{eq3.29}), hence finish proof of (\ref{eq3.26}).
\end{proof}

\newtheorem{coro}{Corollary}[section]
\begin{coro}
Let $({x^*},{y^*})$  denote a saddle point of (\ref{eq1.1}). For parameter setting in \ref{theorem3.2}, it holds that	
\begin{equation}
||{\bar y_{k + 1}} - {y^*}|| \le O(1/k)
\label{eq3.30}
\end{equation}
\end{coro}
\begin{proof}
 In (\ref{eq3.21}) of Theorem 3.2, letting $x = {x^*}$,$y = {y^*}$ , we have
 \[E_{k + 1}^* \le \frac{{(2{L_f} + \tau ||A|{|^2})||{x^*} - {x_1}|{|^2}}}{{k(k + 1)}} + \frac{{||{y^*} - {y_1}|{|^2}}}{{k(k + 1)\tau }}\]
Using (\ref{eq3.26}) and above inequality, we have 
 \[\frac{{{\mu _g}}}{2}||{\bar y_{k + 1}} - {y^*}|{|^2} \le E_{k + 1}^* \le \frac{1}{{k(k + 1)}}\left[ {(2{L_f} + \tau ||A|{|^2})||{x^*} - {x_1}|{|^2} + \frac{{||{y^*} - {y_1}|{|^2}}}{\tau }} \right]\]
 which can be written as (\ref{eq3.30}).
\end{proof}

\subsection{Full Acceleration by Smooth and Strongly Convex Primal Component}
Here let us recover the full acceleration \cite{Xu2016Accelerated} by smooth and strongly convex primal component.
\begin{theorem}
For strongly convex saddle point problem (\ref{eq1.1}) with ${\mu _g} = 0,{\mu _f} > 0$, Let
\begin{equation}
{\theta _t} = 1,\begin{array}{*{20}{c}}
{}
\end{array}{\alpha _t} = \frac{{t + {t_0}}}{{t + {t_0} + 1}},\begin{array}{*{20}{c}}
{}
\end{array}\tau  \le \frac{{{\mu _f}}}{{2||A|{|^2}}},\begin{array}{*{20}{c}}
{}
\end{array}{\tau _t} = (t + 1)\tau ,\begin{array}{*{20}{c}}
{}
\end{array}{\eta _t} = \frac{1}{{{L_f} + {\tau _t}||A|{|^2}}}
\label{eq3.31}
\end{equation}

and ${y_0} = {y_1}$, it holds for LDPD that
\begin{equation}
{\tilde E_{k + 1}} \le \frac{{({t_0} + 2)}}{{k(k + 3 + 2{t_0})}}\left[ {||x - {x_1}|{|^2}\left( {{L_f} - {\mu _f} + 2{\tau _{}}||A|{|^2}} \right) + ||y - {y_1}|{|^2}/(2\tau )} \right]
\label{eq3.32}
\end{equation}
\begin{equation}
{\tilde E_{k + 1}} = L({\tilde x_{k + 1}},y) - L(x,{\tilde y_{k + 1}}),\begin{array}{*{20}{c}}
{}
\end{array}{\tilde x_{k + 1}} = \sum\limits_{t = 1}^k {(t + {t_0} + 1){x_{t + 1}}} /\sum\limits_{t = 1}^k {(t + {t_0} + 1)} 
\label{eq3.33}
\end{equation}
\begin{equation}
{\tilde y_{k + 1}} = \sum\limits_{t = 1}^k {(t + {t_0} + 1){y_{t + 1}}} /\sum\limits_{t = 1}^k {(t + {t_0} + 1)} ,\begin{array}{*{20}{c}}
{}
\end{array}{t_0} = \left\lceil {2\frac{{{L_f} - {\mu _f}}}{{{\mu _f}}}} \right\rceil 
\label{eq3.34}
\end{equation}
\end{theorem}
\begin{proof}
Using ${\theta _t} = 1$, $1/{\eta _t} = ({L_f} + {\tau _t}||A|{|^2})$ and 
${\tau _t} = (t + 1)\tau$ of (\ref{eq3.31}) into (\ref{eq3.10}), we have
\[\begin{array}{l}
{E_{t + 1}} \le {L_f}||{x_{t + 1}} - {x_t}|{|^2}/2 - {\mu _f}||x - {x_t}|{|^2}/2\\
+ ({L_f} + {\tau _t}||A|{|^2})\left( {||x - {x_t}|{|^2} - ||x - {x_{t + 1}}|{|^2} - ||{x_{t + 1}} - {x_t}|{|^2}} \right)/2\\
+ (||y - {y_t}|{|^2} - ||y - {y_{t + 1}}|{|^2} - ||{y_t} - {y_{t + 1}}{\rm{|}}{{\rm{|}}^2})/[2(t + 1)\tau ]\\
+ \left\langle {Ax - A{x_{t + 1}},{y_t} - {y_{t + 1}}} \right\rangle  - \left\langle {Ax - A{x_t},{y_{t - 1}} - {y_t}} \right\rangle {\alpha _t} + \left\langle {A{x_t} - A{x_{t + 1}},{y_t} - {y_{t - 1}}} \right\rangle {\alpha _t}
\end{array}\]
Multiplying above inequality with $(t + {t_0} + 1)$ and using 
$(t + {t_0} + 1){\alpha _t} = (t + {t_0})$ leads to
\begin{equation}
\begin{array}{l}
(t + {t_0} + 1){E_{t + 1}} \le ({L_f} + {\tau _t}||A|{|^2})(t + {t_0} + 1)\left( {||x - {x_t}|{|^2} - ||x - {x_{t + 1}}|{|^2}} \right)/2\\
- (t + {t_0} + 1){\mu _f}||x - {x_t}|{|^2}/2 - (t + {t_0} + 1){\tau _t}||A|{|^2}||{x_{t + 1}} - {x_t}|{|^2}/2\\
+ (t + {t_0} + 1)(||y - {y_t}|{|^2} - ||y - {y_{t + 1}}|{|^2} - ||{y_t} - {y_{t + 1}}{\rm{|}}{{\rm{|}}^2})/[2(t + 1)\tau ]\\
+ (t + {t_0} + 1)\left\langle {Ax - A{x_{t + 1}},{y_t} - {y_{t + 1}}} \right\rangle  - (t + {t_0})\left\langle {Ax - A{x_t},{y_{t - 1}} - {y_t}} \right\rangle \\
+ (t + {t_0})\left\langle {A{x_t} - A{x_{t + 1}},{y_t} - {y_{t - 1}}} \right\rangle 
\end{array}
\label{eq3.35}
\end{equation}
Since ${\tau _t} = (t + 1)\tau$ , we have  
$(t + {t_0}){\tau _{t - 1}} \le (t + {t_0} + 1){\tau _t}$, hence
\begin{equation}
\begin{array}{l}
(t + {t_0})\left\langle {A{x_t} - A{x_{t + 1}},{y_t} - {y_{t - 1}}} \right\rangle \\
\le (t + {t_0})(||A{x_t} - A{x_{t + 1}}|{|^2}{\tau _{t - 1}} + ||{y_t} - {y_{t - 1}}|{|^2}/{\tau _{t - 1}})/2\\
\le (t + {t_0} + 1){\tau _t}||A{x_t} - A{x_{t + 1}}|{|^2}/2 + (t + {t_0})||{y_t} - {y_{t - 1}}|{|^2}/(2{\tau _{t - 1}})
\label{eq3.36}
\end{array}
\end{equation}
Summing up (\ref{eq3.35}) over $t=1,2,...k$ , and using (\ref{eq3.36}), we have 
\begin{equation}
\begin{array}{l}
\sum\limits_{t = 1}^k {(t + {t_0} + 1){E_{t + 1}}}  \le \sum\limits_{t = 1}^k {\left[ {({L_f} + {\tau _t}||A|{|^2})(t + {t_0} + 1) - (t + {t_0} + 1){\mu _f}} \right]\frac{{||x - {x_t}|{|^2}}}{2}} \\
- \sum\limits_{t = 1}^k {({L_f} + {\tau _t}||A|{|^2})(t + {t_0} + 1)\frac{{||x - {x_{t + 1}}|{|^2}}}{2}}  - \sum\limits_{t = 1}^k {(t + {t_0} + 1){\tau _t}||A|{|^2}\frac{{||{x_{t + 1}} - {x_t}|{|^2}}}{2}} \\
+ ({t_0} + 2)(||y - {y_1}|{|^2})/(4\tau ) - \sum\limits_{t = 1}^k {(t + {t_0} + 1)||{y_t} - {y_{t + 1}}{\rm{|}}{{\rm{|}}^2}/[2(t + 1)\tau ]} \\
+ (k + {t_0} + 1)\left\langle {Ax - A{x_{k + 1}},{y_k} - {y_{k + 1}}} \right\rangle \\
+ \sum\limits_{t = 1}^k {\left[ {(t + {t_0} + 1){\tau _t}||A{x_t} - A{x_{t + 1}}|{|^2}/2 + (t + {t_0})||{y_t} - {y_{t - 1}}|{|^2}/(2{\tau _{t - 1}})} \right]} 
\end{array}
\label{eq3.37}
\end{equation}
Noting that
\[\begin{array}{l}
2(k + {t_0} + 1)\left\langle {Ax - A{x_{k + 1}},{y_k} - {y_{k + 1}}} \right\rangle \\
\le (k + {t_0} + 1)2||{y_k} - {y_{k + 1}}||(||A|| \cdot ||x - {x_{k + 1}}||)\\
\le (k + {t_0} + 1)\left[ {||{y_k} - {y_{k + 1}}|{|^2}/{\tau _k} + {\tau _k}||A|{|^2}||x - {x_{k + 1}}|{|^2}} \right]\\
= (k + {t_0} + 1)||{y_k} - {y_{k + 1}}|{|^2}/{\tau _k} + (k + {t_0} + 1){\tau _k}||A|{|^2}||x - {x_{k + 1}}|{|^2}
\end{array}\]
We can simplify (\ref{eq3.37}) as	
\begin{equation}
\begin{array}{l}
2\sum\limits_{t = 1}^k {(t + {t_0} + 1){E_{t + 1}}}  \le \sum\limits_{t = 1}^k {\left[ {({L_f} + {\tau _t}||A|{|^2})(t + {t_0} + 1) - (t + {t_0} + 1){\mu _f}} \right]||x - {x_t}|{|^2}} \\
- \sum\limits_{t = 1}^k {\left[ {({L_f} + {\tau _t}||A|{|^2})(t + {t_0} + 1)||x - {x_{t + 1}}|{|^2}} \right]}  + ({t_0} + 2)(||y - {y_1}|{|^2})/(2\tau )\\
+ (k + {t_0} + 1){\tau _k}||A|{|^2}||x - {x_{k + 1}}|{|^2}
\end{array}
\label{eq3.38}
\end{equation}
Suppose condition 
\begin{equation}
({L_f} + {\tau _t}||A|{|^2})(t + {t_0} + 1) \ge ({L_f} + {\tau _{t + 1}}||A|{|^2})(t + {t_0} + 2) - (t + {t_0} + 2){\mu _f}
\label{eq3.39}
\end{equation}
is satisfied, we can further simplify (\ref{eq3.38}) as
\begin{equation}
\begin{array}{l}
2\sum\limits_{t = 1}^k {(t + {t_0} + 1){E_{t + 1}}}  \le ({t_0} + 2)(||y - {y_1}|{|^2})/(2\tau )\\
+ \left[ {({L_f} + 2\tau ||A|{|^2})({t_0} + 2) - ({t_0} + 2){\mu _f}} \right]||x - {x_1}|{|^2}
\end{array}
\end{equation}
Using above inequality and Jensen inequality, we have (\ref{eq3.32}).
condition (\ref{eq3.39}) is equivalent to
\[({\mu _f} - \tau ||A|{|^2})(t + {t_0}) \ge ({L_f} + \tau ||A|{|^2} - {\mu _f}) + 2\tau ||A|{|^2} - {\mu _f}\]
and can be satisfied by $2\tau ||A|{|^2} - {\mu _f} \le 0$ in 

(\ref{eq3.31})  and ${t_0} \ge 2({L_f} - {\mu _f})/{\mu _f}$ from (\ref{eq3.34}).
\end{proof}

\begin{remark}\label{remark3.1}
For single step gradient LDPD with weakly convex problem, one can let
\[{\theta _t} = 1,\begin{array}{*{20}{c}}
{}
\end{array}{\alpha _t} = 1,\begin{array}{*{20}{c}}
{}
\end{array}{\mu _f} = {\mu _g} = 0,\begin{array}{*{20}{c}}
{}
\end{array}{\tau _t} = \tau  > 0,\begin{array}{*{20}{c}}
{}
\end{array}{\eta _t} = \eta  = 1/({L_f} + \tau ||A|{|^2})\]	
Using Jensen inequality, we have  
${\tilde E_{k + 1}} \le \frac{1}{{2k}}({L_f} + \tau ||A|{|^2})||x - {x_1}|{|^2} + \frac{1}{{2k\tau }}||y - {y_1}|{|^2}$
where ${\tilde E_{k + 1}} = L({\tilde x_{k + 1}},y) - L(x,{\tilde y_{k + 1}})$
, ${\tilde x_{k + 1}} = (1/k)\sum\limits_{t = 1}^k {{x_{t + 1}}}$ and 
${\tilde y_{k + 1}} = (1/k)\sum\limits_{t = 1}^k {{y_{t + 1}}}$ . \\
The proof is relatively simpler and omitted.
\end{remark}

\section{Exact DPD(EDPD)}
In this section, we present Exact DPD(EDPD) which does not require linearization of primal component hence can be directly apply to solve non-smooth convex concave point problem. 

\subsection{Iteration of EDPD}
The iteration of EDPD is listed as following:\\
\noindent\rule{17cm}{0.05em}\\
  		\textbf{For $t=1,2,...,k$   }
  		\vspace{-5mm}
		\begin{align}
			{x_{t + 1}}&= \arg {\min _x}f(x) + \left\langle {Ax,{{\hat y}_t}} \right\rangle  + \frac{1}{{2{\eta _t}}}||x - {x_t}|{|^2} \label{eq4.1} \\
			{y_{t + 1}}&= \arg {\min _y}||y - ({y_t} + {\tau _t}A{x_{t + 1}})|{|^2}/(2{\tau _t}) + g(y)\label{eq4.2}\\
			{\hat y_{t + 1}}&= ({y_{t + 1}} - {y_t}){\alpha _{t + 1}} + {y_{t + 1}}\label{eq4.3}
		\end{align}
		\textbf{EndFor}\\
\noindent\rule{17cm}{0.05em}\\
\vspace{-7mm}

   Note that LDPD use linearization and multi-step gradient method to update primal variable as (\ref{eq2.2}), whereas EDPD does not linearize f(x) and updates primal variable as (\ref{eq4.1}).

\subsection{Gap Function and Optimal Condition}
Let us firstly define gap function with respect to iteration
$({x_{t + 1}},{y_{t + 1}})$ as
\begin{equation}
\begin{array}{l}
{E_{t + 1}} = L({x_{t + 1}},y) - L(x,{y_{t + 1}})\\
{\rm{ = }}f({x_{t + 1}}) - f(x){\rm{ + }}\left\langle {A{x_{t + 1}},y} \right\rangle  - \left\langle {Ax,{y_{t + 1}}} \right\rangle  - \left[ {g(y) - g({y_{t + 1}})} \right]
\label{eq4.4}
\end{array}
\end{equation}
Optimal condition of exact update (\ref{eq4.1}) is 
\[0 \in \partial f({x_{t + 1}}) + {A^T}{\hat y_t} + ({x_{t + 1}} - {x_t})/{\eta _t}\]
Since $f(x)$ is strongly convex with parameter ${\mu _f}$, there is
\[f(x) - f({x_{t + 1}}) \ge (x - {x_{t + 1}})\partial f({x_{t + 1}}) + {\mu _f}||x - {x_{t + 1}}|{|^2}/2\]
Combining above two relation, we have
\begin{equation}
f({x_{t + 1}}) - f(x) \le \left\langle {x - {x_{t + 1}},{x_{t + 1}} - {x_t}} \right\rangle /{\eta _t} + \left\langle {Ax - A{x_{t + 1}},{{\hat y}_t}} \right\rangle  - {\mu _f}||x - {x_{t + 1}}|{|^2}/2
\label{eq4.5}
\end{equation}

\begin{prop}
It holds for iteration in EDPD that
\begin{equation}
\begin{array}{l}
{E_{t + 1}} \le \left\langle {x - {x_{t + 1}},{x_{t + 1}} - {x_t}} \right\rangle /{\eta _t} - {\mu _f}||x - {x_{t + 1}}|{|^2}/2\\
+ \left\langle {y - {y_{t + 1}},{y_{t + 1}} - {y_t}} \right\rangle /{\tau _t} - {\mu _g}||y - {y_{t + 1}}|{|^2}/2\\
+ \left\langle {Ax - A{x_{t + 1}},{y_t} - {y_{t + 1}}} \right\rangle  - \left\langle {Ax - A{x_t},{y_{t - 1}} - {y_t}} \right\rangle {\alpha _t}\\
+ \left\langle {A{x_t} - A{x_{t + 1}},{y_t} - {y_{t - 1}}} \right\rangle {\alpha _t}
\end{array}
\label{eq4.6}
\end{equation}
\end{prop}
\begin{proof}
 Using the notation of $L(x,y)$  in (\ref{eq1.1}), we have
 \begin{equation}
 \begin{array}{l}
 {E_{t + 1}} = L({x_{t + 1}},y) - L(x,{y_{t + 1}})\\
 {\rm{ = }}f({x_{t + 1}}) - f(x){\rm{ + }}\left\langle {A{x_{t + 1}},y} \right\rangle  - \left\langle {Ax,{y_{t + 1}}} \right\rangle  + g({y_{t + 1}}) - g(y)
 \end{array}
 \label{eq4.7}
 \end{equation}
Note that the dual minimization (\ref{eq4.2}) is the same as (\ref{eq2.4}), we still can apply (\ref{eq3.1}).
Applying (\ref{eq4.5}) and (\ref{eq3.1}) into (\ref{eq4.7}) yields
\begin{equation}
\begin{array}{l}
{E_{t + 1}} \le \left\langle {x - {x_{t + 1}},{x_{t + 1}} - {x_t}} \right\rangle /{\eta _t} + \left\langle {Ax - A{x_{t + 1}},{{\hat y}_t}} \right\rangle  - {\mu _f}||x - {x_{t + 1}}|{|^2}/2\\
{\rm{ + }}\left\langle {A{x_{t + 1}},y} \right\rangle  - \left\langle {Ax,{y_{t + 1}}} \right\rangle  + \left\langle {y - {y_{t + 1}},{y_{t + 1}} - {y_t}} \right\rangle /{\tau _t}\\
- {\mu _g}||y - {y_{t + 1}}|{|^2}/2 - \left\langle {A{x_{t + 1}},y - {y_{t + 1}}} \right\rangle 
\end{array}
\label{eq4.8}
\end{equation}
Since dual extrapolation (\ref{eq4.3}) is the same as (\ref{eq2.5}), we have the same equality as (\ref{eq3.12}), written as

\[\begin{array}{l}
\left\langle {Ax - A{x_{t + 1}},{{\hat y}_t}} \right\rangle {\rm{ + }}\left\langle {A{x_{t + 1}},y} \right\rangle  - \left\langle {Ax,{y_{t + 1}}} \right\rangle  - \left\langle {A{x_{t + 1}},y - {y_{t + 1}}} \right\rangle \\
= \left\langle {Ax - A{x_{t + 1}},{y_t} - {y_{t + 1}}} \right\rangle  - \left\langle {Ax - A{x_t},{y_{t - 1}} - {y_t}} \right\rangle {\alpha _t} + \left\langle {A{x_t} - A{x_{t + 1}},{y_t} - {y_{t - 1}}} \right\rangle {\alpha _t}
\end{array}\]
Substituting above equality into (\ref{eq4.8}), we have (\ref{eq4.6}).
\end{proof}
\subsection{ EDPD with Strongly Convex Primal Component}
Using (\ref{eq4.6}), we can derive the convergence rate for Exact DPD in Section 4.1. 

\begin{theorem}
 For strongly convex saddle point problem (\ref{eq1.1}) with
 ${\mu _g} = 0$ and ${\mu _f} > 0$ ,\\ Let  $\{ {x_{t + 1}},{y_{t + 1}}\} $
 be produced by EDPD in (\ref{eq4.1})-(\ref{eq4.3}), let dual step-size ${\tau _t}$, primal step-size ${\eta _t}$ and
  extrapolation parameter ${\alpha _{t + 1}}$  be set as
 \begin{equation}
 {\tau _t} = (t + 1)\tau ,\begin{array}{*{20}{c}}
 {}
 \end{array}\tau  \le \frac{{{\mu _f}}}{{2||A|{|^2}}},\begin{array}{*{20}{c}}
 {}
 \end{array}{\eta _t} = \frac{1}{{{\tau _t}||A|{|^2}}},\begin{array}{*{20}{c}}
 {}
 \end{array}{\alpha _{t + 1}} = \frac{{t + 2}}{{t + 3}}
 \label{eq4.9}
  \end{equation}
  It holds for EDPD that\\
  \begin{equation}
  {\tilde E_{k + 1}} \le \frac{1}{{k(k + 5)}}\left( {6\tau ||A|{|^2}||x - {x_1}|{|^2} + \frac{3}{{2\tau }}||y - {y_1}|{|^2}} \right)
  \label{eq4.10}
  \end{equation}
  with 
  ${\tilde E_{k + 1}} = L({\tilde x_{k + 1}},y) - L(x,{\tilde y_{k + 1}})$
  and
  \begin{equation}
  {\tilde x_{k + 1}} = \sum\limits_{t = 1}^k {(t + 2){x_{t + 1}}} /\left( {\sum\limits_{t = 1}^k {(t + 2)} } \right),\begin{array}{*{20}{c}}
  {}
  \end{array}{\tilde y_{k + 1}} = \sum\limits_{t = 1}^k {(t + 2){y_{t + 1}}} /\left( {\sum\limits_{t = 1}^k {(t + 2)} } \right)
  \label{eq4.11}
   \end{equation}
 \end{theorem}
  \begin{proof}
   Plugging ${\tau _t}$, ${\eta _t}$  in (\ref{eq4.9}) and ${\mu _g} = 0$   into (\ref{eq4.6}), we have
   \begin{equation}
   \begin{array}{l}
   2{E_{t + 1}} \le (t + 1)\tau ||A|{|^2}(||x - {x_t}|{|^2} - ||x - {x_{t + 1}}|{|^2} - ||{x_{t + 1}} - {x_t}|{|^2})\\
   - {\mu _f}||x - {x_{t + 1}}|{|^2} + (||y - {y_t}|{|^2} - ||y - {y_{t + 1}}|{|^2} - ||{y_{t + 1}} - {y_t}|{|^2})/{\tau _t}\\
   + 2\left\langle {Ax - A{x_{t + 1}},{y_t} - {y_{t + 1}}} \right\rangle  - 2\left\langle {Ax - A{x_t},{y_{t - 1}} - {y_t}} \right\rangle {\alpha _t}\\
   + 2\left\langle {A{x_t} - A{x_{t + 1}},{y_t} - {y_{t - 1}}} \right\rangle {\alpha _t}
   \end{array}
   \end{equation}
  Multiplying above inequality with $(t + 2)$
     and using $(t + 2){\alpha _t} = (t + 1)$ from (\ref{eq4.9}), we have 
     \begin{equation}
     \begin{array}{l}
     2(t + 2){E_{t + 1}} \le (t + 2)(t + 1)\tau ||A|{|^2}(||x - {x_t}|{|^2} - ||x - {x_{t + 1}}|{|^2} - ||{x_{t + 1}} - {x_t}|{|^2})\\
     - (t + 2){\mu _f}||x - {x_{t + 1}}|{|^2} + \frac{{t + 2}}{{t + 1}}(||y - {y_t}|{|^2} - ||y - {y_{t + 1}}|{|^2} - ||{y_{t + 1}} - {y_t}|{|^2})/\tau \\
     + 2(t + 2)\left\langle {Ax - A{x_{t + 1}},{y_t} - {y_{t + 1}}} \right\rangle  - 2\left\langle {Ax - A{x_t},{y_{t - 1}} - {y_t}} \right\rangle (t + 1)\\
     + 2\left\langle {A{x_t} - A{x_{t + 1}},{y_t} - {y_{t - 1}}} \right\rangle (t + 1)
     \end{array}
    \label{eq4.13} 
     \end{equation}
    Noting that
    \begin{equation}
   \begin{array}{l}
    2(t + 1)\left\langle {A{x_t} - A{x_{t + 1}},{y_t} - {y_{t - 1}}} \right\rangle \\
    \le 2(t + 1)||A{x_t} - A{x_{t + 1}}|| \cdot ||{y_t} - {y_{t - 1}}||\\
    \le (t + 1)\left( {||A{x_t} - A{x_{t + 1}}|{|^2}(t + 2)\tau  + ||{y_t} - {y_{t - 1}}|{|^2}/[(t + 2)\tau ]} \right)\\
    \le (t + 1)\left( {||A|{|^2}||{x_t} - {x_{t + 1}}|{|^2}(t + 2)\tau  + ||{y_t} - {y_{t - 1}}|{|^2}/[(t + 2)\tau ]} \right)
    \end{array}
    \label{eq4.14}
    \end{equation} 
    Summing up (\ref{eq4.13}) over $t=1,2,...k$  with
    $\frac{{t + 2}}{{t + 1}} > \frac{{t + 3}}{{t + 2}}$,
    ${y_1} - {y_0} = 0$ and using (\ref{eq4.14}), we have
\[\begin{array}{l}
2\sum\limits_{t = 1}^k {(t + 2){E_{t + 1}}} \\
\le \sum\limits_{t = 1}^k {\left[ {(t + 2)(t + 1)\tau ||A|{|^2}(||x - {x_t}|{|^2} - ||x - {x_{t + 1}}|{|^2}) - (t + 2){\mu _f}||x - {x_{t + 1}}|{|^2}} \right]} \\
- \sum\limits_{t = 1}^k {(t + 2)(t + 1)\tau ||A|{|^2}||{x_{t + 1}} - {x_t}|{|^2}}  + 2(k + 2)\left\langle {Ax - A{x_{k + 1}},{y_k} - {y_{k + 1}}} \right\rangle \\
+ \left( {\frac{3}{2}||y - {y_1}|{|^2} - \frac{{k + 2}}{{k + 1}}||y - {y_{k + 1}}|{|^2} - \sum\limits_{t = 1}^k {\frac{{t + 2}}{{t + 1}}} ||{y_{t + 1}} - {y_t}|{|^2}} \right)/\tau \\
+ \sum\limits_{t = 1}^k {(t + 1)\left( {||A|{|^2}||{x_t} - {x_{t + 1}}|{|^2}(t + 2)\tau  + ||{y_t} - {y_{t - 1}}|{|^2}/[(t + 2)\tau ]} \right)} 
\end{array}\] 
     The above inequality can be simplified as
     \begin{equation}
     \begin{array}{l}
     2\sum\limits_{t = 1}^k {(t + 2){E_{t + 1}}}  \le \sum\limits_{t = 1}^k {\left[ {(t + 2)(t + 1)\tau ||A|{|^2}(||x - {x_t}|{|^2} - ||x - {x_{t + 1}}|{|^2})} \right]} \\
     - \sum\limits_{t = 1}^k {(t + 2){\mu _f}||x - {x_{t + 1}}|{|^2}}  + 2(k + 2)\left\langle {Ax - A{x_{k + 1}},{y_k} - {y_{k + 1}}} \right\rangle \\
     + \left( {\frac{3}{2}||y - {y_1}|{|^2} - \frac{{k + 2}}{{k + 1}}||y - {y_{k + 1}}|{|^2} - \frac{{k + 2}}{{k + 1}}||{y_{k + 1}} - {y_k}|{|^2}} \right)/\tau 
     \end{array}
     \label{eq4.15}
     \end{equation}
     Note that
     \begin{equation}
     \begin{array}{l}
     2(k + 2)\left\langle {Ax - A{x_{k + 1}},{y_k} - {y_{k + 1}}} \right\rangle \\
     \le 2(k + 2)||Ax - A{x_{k + 1}}|| \cdot ||{y_k} - {y_{k + 1}}||\\
     \le (k + 2)((k + 1)||Ax - A{x_{k + 1}}|{|^2} + ||{y_k} - {y_{k + 1}}|{|^2}/(k + 1))\\
     \le (k + 2)\left\{ {[(k + 1)\tau ]||A|{|^2}||x - {x_{k + 1}}|{|^2} + ||{y_k} - {y_{k + 1}}|{|^2}/[\tau (k + 1)]} \right\}
     \end{array}
     \label{eq4.16}
     \end{equation}
     
     Plugging (\ref{eq4.16}) into (\ref{eq4.15}) leads to
     \begin{equation}
     \begin{array}{l}
     2\sum\limits_{t = 1}^k {(t + 2){E_{t + 1}}}  \le \sum\limits_{t = 1}^k {\left[ {(t + 2)(t + 1)\tau ||A|{|^2}(||x - {x_t}|{|^2} - ||x - {x_{t + 1}}|{|^2})} \right]} \\
     - \sum\limits_{t = 1}^k {(t + 2){\mu _f}||x - {x_{t + 1}}|{|^2}}  + \left( {\frac{3}{2}||y - {y_1}|{|^2} - \frac{{k + 2}}{{k + 1}}||y - {y_{k + 1}}|{|^2}} \right)/\tau \\
     + (k + 2)(k + 1)\tau ||A|{|^2}||x - {x_{k + 1}}|{|^2}
     \end{array}
     \label{eq4.17}
     \end{equation}
     Suppose condition
     \begin{equation}
     (t + 2)(t + 1)\tau ||A|{|^2} + (t + 2){\mu _f} \ge (t + 3)(t + 2)\tau ||A|{|^2}
     \label{eq4.18}
     \end{equation}
     is satisfied.
     we have
     \[\begin{array}{l}
     2\sum\limits_{t = 1}^k {(t + 2){E_{t + 1}}}  \le 6\tau ||A|{|^2}||x - {x_1}|{|^2} - (k + 2)(k + 1)\tau ||A|{|^2}||x - {x_{k + 1}}|{|^2}\\
     + \left( {\frac{3}{2}||y - {y_1}|{|^2} - \frac{{k + 2}}{{k + 1}}||y - {y_{k + 1}}|{|^2}} \right)/\tau  + (k + 2)(k + 1)\tau ||A|{|^2}||x - {x_{k + 1}}|{|^2}
     \end{array}\]
     which leads to
     \[\sum\limits_{t = 1}^k {(t + 2){E_{t + 1}}}  \le 3\tau ||A|{|^2}||x - {x_1}|{|^2} + \frac{3}{{4\tau }}||y - {y_1}|{|^2}\]
     Using above inequality and Jensen inequality, we have (\ref{eq4.10}). Condition (\ref{eq4.18}) can be simplified as 
     $(t + 2){\mu _f} \ge (2t + 4)\tau ||A|{|^2}$,
     which is satisfied by ${\mu _f}/2 \ge \tau ||A|{|^2}$  from (\ref{eq4.9}).
  \end{proof}

 \begin{remark}\label{remark4.1}
Applying EDPD without strongly convex component to weakly convex
saddle point problem (\ref{eq1.1})  with ${\mu _f} = {\mu _g} = 0$, 
and letting dual extrapolation parameter ${\alpha _t} = 1$, dual step size ${\tau _t} = \tau  > 0$ and primal step size
 ${\eta _t} = \eta  \le 1/(\tau ||A|{|^2})$, then we have
 \[{\tilde E_{k + 1}} \le \frac{1}{{2k}}\left( {||x - {x_1}|{|^2}||A|{|^2}\tau  + ||y - {y_1}|{|^2}/\tau } \right)\]
 With  ${\tilde E_{k + 1}} = L({\tilde x_{k + 1}},y) - L(x,{\tilde y_{k + 1}})$, ${\tilde x_{k + 1}} = (1/k)\sum\limits_{t = 1}^k {{x_{t + 1}}}$
 and  ${\tilde y_{k + 1}} = (1/k)\sum\limits_{t = 1}^k {{y_{t + 1}}}$.
 The proof is relatively simpler and omitted.
 
 \end{remark}

\subsection{EDPD with Strongly Convex Dual component }\label{section4.4}  

\begin{theorem}
	For strongly convex saddle point problem (\ref{eq1.1}) with 
	${\mu _g} > 0, {\mu _f} = 0$, Let
	$\{ {x_{t + 1}},{y_{t + 1}}\} $
	be produced by EDPD (\ref{eq4.1})-(\ref{eq4.3}), and extrapolation parameter 
	${\alpha _{t + 1}}$, dual step-size  ${\tau _t}$ and primal step-size ${\eta _t}$ be set as
	\begin{equation}
	{\alpha _{t + 1}} = \frac{{t + 1}}{{t + 2}},\begin{array}{*{20}{c}}
	{}
	\end{array}\tau  = \frac{{2.5}}{{{\mu _g}}},\begin{array}{*{20}{c}}
	{}
	\end{array}{\tau _t} = \frac{\tau }{{t + 1}},\begin{array}{*{20}{c}}
	{}
	\end{array}{\eta _t} = \frac{{(t + 1)}}{{\tau ||A|{|^2}}},\begin{array}{*{20}{c}}
	{}
	\end{array}{y_0} = {y_1}
	\label{eq4.19}
	\end{equation}
	It holds for EDPD that
	\begin{equation}
	{\tilde E_{k + 1}} \le \frac{2}{{k(k + 3)}}\left[ {||A|{|^2}\tau ||x - {x_1}|{|^2}/2 + ||y - {y_1}|{|^2}2/\tau } \right]
	\label{eq4.20}
	\end{equation}
	where
	\begin{equation}
	{\tilde E_{k + 1}} = L({\tilde x_{k + 1}},y) - L(x,{\tilde y_{k + 1}}),\begin{array}{*{20}{c}}
	{}
	\end{array}{\tilde x_{k + 1}} = \sum\limits_{t = 1}^k {(t + 1){x_{t + 1}}} /\sum\limits_{t = 1}^k {(t + 1)} 
	\label{eq4.21}
	\end{equation}
	and
	\begin{equation}
	{\tilde y_{k + 1}} = \sum\limits_{t = 1}^k {(t + 1){y_{t + 1}}} /\sum\limits_{t = 1}^k {(t + 1)} 
	\label{eq4.22}
	\end{equation}
	
\end{theorem}
\begin{proof}
Plugging  ${\mu _f} = 0$  into (\ref{eq4.6}), then we have

\[\begin{array}{l}
{E_{t + 1}} \le \left( {||x - {x_t}|{|^2} - ||x - {x_{t + 1}}|{|^2} - ||{x_{t + 1}} - {x_t}|{|^2}} \right)/(2{\eta _t})\\
+ \left( {||y - {y_t}|{|^2} - ||y - {y_{t + 1}}|{|^2} - ||{y_{t + 1}} - {y_t}|{|^2}} \right)/(2{\tau _t}) - {\mu _g}||y - {y_{t + 1}}|{|^2}/2\\
+ \left\langle {Ax - A{x_{t + 1}},{y_t} - {y_{t + 1}}} \right\rangle  - \left\langle {Ax - A{x_t},{y_{t - 1}} - {y_t}} \right\rangle {\alpha _t}\\
+ \left\langle {A{x_t} - A{x_{t + 1}},{y_t} - {y_{t - 1}}} \right\rangle {\alpha _t}
\end{array}\]
Multiplying above inequality with  $2(t + 1)$
 and using $1/{\eta _t} = \tau ||A|{|^2}/(t + 1)$,
 $1/{\tau _t} = (t + 1)/\tau$,
 $(t + 1){\alpha _t} = t$  from (\ref{eq4.19}), we have
 \begin{equation}
 \begin{array}{l}
 2(t + 1){E_{t + 1}} \le \tau ||A|{|^2}\left( {||x - {x_t}|{|^2} - ||x - {x_{t + 1}}|{|^2} - ||{x_{t + 1}} - {x_t}|{|^2}} \right)\\
 + (t + 1)\left( {||y - {y_t}|{|^2} - ||y - {y_{t + 1}}|{|^2} - ||{y_{t + 1}} - {y_t}|{|^2}} \right)(t + 1)/\tau \\
 - (t + 1){\mu _g}||y - {y_{t + 1}}|{|^2} + 2(t + 1)\left\langle {Ax - A{x_{t + 1}},{y_t} - {y_{t + 1}}} \right\rangle \\
 - 2t\left\langle {Ax - A{x_t},{y_{t - 1}} - {y_t}} \right\rangle  + 2t\left\langle {A{x_t} - A{x_{t + 1}},{y_t} - {y_{t - 1}}} \right\rangle 
 \end{array}
 \label{eq4.23} 
 \end{equation}
 Summing up (\ref{eq4.23}) over $t=1,2,...k$ with
 ${y_1} - {y_0} = 0$, we have
 \begin{equation}
  \begin{array}{l}
 2\sum\limits_{t = 1}^k {(t + 1){E_{t + 1}}}  \le ||A|{|^2}\tau ||x - {x_1}|{|^2} - ||A|{|^2}\tau ||x - {x_{k + 1}}|{|^2}\\
 - \sum\limits_{t = 1}^k {||A|{|^2}\tau ||{x_{t + 1}} - {x_t}|{|^2}}  - \sum\limits_{t = 1}^k {(t + 1){\mu _g}||y - {y_{t + 1}}|{|^2}} \\
 + \sum\limits_{t = 1}^k {(t + 1)\left( {||y - {y_t}|{|^2} - ||y - {y_{t + 1}}|{|^2} - ||{y_{t + 1}} - {y_t}|{|^2}} \right)(t + 1)/\tau } \\
 + 2(k + 1)\left\langle {Ax - A{x_{k + 1}},{y_k} - {y_{k + 1}}} \right\rangle  + \sum\limits_{t = 1}^k {2t\left\langle {A{x_t} - A{x_{t + 1}},{y_t} - {y_{t - 1}}} \right\rangle } 
 \end{array}
 \label{eq4.24}
 \end{equation}
 Suppose condition
 \begin{equation}
 {(t + 1)^2}/\tau  + (t + 1){\mu _g} \ge {(t + 2)^2}/\tau 
 \label{eq4.25}
 \end{equation}

 is satisfied.
 Using the above condition to (\ref{eq4.24}), then we have
 \begin{equation}
 \begin{array}{l}
 2\sum\limits_{t = 1}^k {(t + 1){E_{t + 1}}}  \le ||A|{|^2}\tau ||x - {x_1}|{|^2} - ||A|{|^2}\tau ||x - {x_{k + 1}}|{|^2}\\
 - \sum\limits_{t = 1}^k {||A|{|^2}\tau ||{x_{t + 1}} - {x_t}|{|^2}}  - \sum\limits_{t = 1}^k {{{(t + 1)}^2}||{y_{t + 1}} - {y_t}|{|^2}/\tau } \\
 + ||y - {y_1}|{|^2}4/\tau  - ||y - {y_{k + 1}}|{|^2}{(k + 1)^2}/\tau  - (k + 1){\mu _g}||y - {y_{k + 1}}|{|^2}\\
 + 2(k + 1)\left\langle {Ax - A{x_{k + 1}},{y_k} - {y_{k + 1}}} \right\rangle  + \sum\limits_{t = 1}^k {2t\left\langle {A{x_t} - A{x_{t + 1}},{y_t} - {y_{t - 1}}} \right\rangle } 
 \end{array}
 \label{eq4.26}
 \end{equation}
 Noting the last term in above inequality

 \[\begin{array}{l}
 2t\left\langle {A{x_t} - A{x_{t + 1}},{y_t} - {y_{t - 1}}} \right\rangle \\
 \le 2||A{x_t} - A{x_{t + 1}}||||{y_t} - {y_{t - 1}}||t\\
 \le ||A|{|^2}||{x_t} - {x_{t + 1}}|{|^2}\tau  + ||{y_t} - {y_{t - 1}}|{|^2}{t^2}/\tau 
 \end{array}\]
 
 Using the above inequality into (\ref{eq4.26}), we have
\begin{equation}
 \begin{array}{l}
2\sum\limits_{t = 1}^k {(t + 1){E_{t + 1}}}  \le ||A|{|^2}\tau ||x - {x_1}|{|^2} - ||A|{|^2}\tau ||x - {x_{k + 1}}|{|^2}\\
- {(k + 1)^2}||{y_{k + 1}} - {y_k}|{|^2}/\tau  + ||y - {y_1}|{|^2}4/\tau  - ||y - {y_{k + 1}}|{|^2}{(k + 1)^2}/\tau \\
- (k + 1){\mu _g}||y - {y_{k + 1}}|{|^2} + 2(k + 1)\left\langle {Ax - A{x_{k + 1}},{y_k} - {y_{k + 1}}} \right\rangle 
\end{array}
\label{eq4.27}
\end{equation}
Noting last term in above inequality

\[\begin{array}{l}
2(k + 1)\left\langle {Ax - A{x_{k + 1}},{y_k} - {y_{k + 1}}} \right\rangle \\
\le 2(k + 1)||A||||x - {x_{k + 1}}||||{y_k} - {y_{k + 1}}||\\
\le ||A|{|^2}||x - {x_{k + 1}}|{|^2}\tau  + {(k + 1)^2}||{y_k} - {y_{k + 1}}|{|^2}/\tau 
\end{array}\]
Plugging above inequality into (\ref{eq4.27}) leads to
\begin{equation}
\sum\limits_{t = 1}^k {(t + 1){E_{t + 1}}}  \le ||A|{|^2}\tau ||x - {x_1}|{|^2}/2 + ||y - {y_1}|{|^2}2/\tau 
\label{eq4.28}
\end{equation}

Condition (\ref{eq4.25}) is equivalent to
$\tau  \ge (2t + 3)/[(t + 1){\mu _g}] = \left( {2 + 1/(t + 1)} \right)/{\mu _g}$, which is satisfied by setting
$\tau  = 2.5/{\mu _g}$  as in (\ref{eq4.19}). Using (\ref{eq4.28}) and Jensen inequality, we have (\ref{eq4.20})-(\ref{eq4.22}). 

\end{proof}

\section{Applications on Image Deblurring}
 In this section, we apply LDPD in section 3 to Gaussian noisy image deblurring and EDPD in section 4 to Salt-Pepper Noisy image deblurring.
 \subsection{Gaussian Noisy Image Deblurring}\label{section5.1}
 Here, we apply LDPD(\ref{eq2.1}-\ref{eq2.6}) to Gaussian noisy image deblurring. Recovering m row by n column grayscale image $x \in {R^{mn \times 1}}$
 from Gaussian noisy and blurred observation   can be written as
 \begin{equation}
 {\min _x}\begin{array}{*{20}{c}}
 {}
 \end{array}\frac{\mu }{2}||Kx - b|{|^2} + ||Dx|{|_{2,1}}
 \label{eq5.1}
 \end{equation}
 where $K \in {R^{mn \times mn}}$  denotes convolution operator and  
$D \in {R^{2mn \times mn}}$ is total difference operator consisting of vertical and horizontal difference operator \cite{Wang2008A} \cite{Xie2017A}. Total variation for image $x \in {R^{mn \times 1}}$
is written as $||Dx|{|_{2,1}}$ . Setting convolution operator $K = I$  will reduce the above formulation to Gaussian denoising. The smoothed total variation version of (\ref{eq5.1}) can be written as saddle point model with strongly convex dual component 
 \begin{equation}
 \mathop {\min }\limits_x \left\{ {\frac{\mu }{2}||Kx - b|{|^2} + \mathop {\max }\limits_y \left( {\left\langle {Dx,y} \right\rangle  - \frac{{{\mu _g}}}{2}||y|{|^2}} \right)} \right\}
 \label{eq5.2}
 \end{equation}
 where the max term denotes smoothed total variation by imposing the strongly convexity into dual component\cite{Goldfarb2013Fast}\cite{Becker2009NESTA}. Apparently 
 ${\mu _g} = 0$
  will deliver the common Total variation 
 $\mathop {\max }\limits_y \left\langle {Dx,y} \right\rangle  = ||Dx|{|_{2,1}}$. The above saddle point problem (\ref{eq5.2}) is a special case of (\ref{eq1.1}) with
 \[f(x) = \frac{\mu }{2}||Kx - b|{|^2},\begin{array}{*{20}{c}}
 {}
 \end{array}A = D,\begin{array}{*{20}{c}}
 {}
 \end{array}g(y) = {\delta _Y}(y) + \frac{{{\mu _g}}}{2}||y|{|^2}\]
 where ${\delta _Y}(y)$ is the indicator function for convex set 
 
 \[Y = \{ y|y \in {R^{2mn \times 1}},||{y_i}|{|_2} \le 1\begin{array}{*{20}{c}}
 {,{y_i} \in {R^2},\begin{array}{*{20}{c}}
 	{i = 1,...,mn}
 	\end{array}}
 \end{array}\} \]
 Update for dual variable $y$ is implemented as 
  ${y_{t + 1}} = {P_{|| \cdot |{|_2} \le 1}}\left[ {({{\bar y}_t} + {\tau _t}D{x_t})/{\rm{(}}{\tau _t}{\mu _g}{\rm{ + 1)}}} \right]$
   where ${P_{|| \cdot |{|_2} \le 1}}$
 denotes projection onto Euclidean Ball. Update for primal variable 
 $x$ uses multi-step gradient $\nabla f({\hat x_t})$ and proceeds as  
 ${x_{t + 1}} = {x_t} - {\eta _t}(\nabla f({\hat x_t}) + {D^T}{y_{t + 1}})$.
 \subsection{Salt-Pepper Noisy Image Deblurring}
Here we apply EDPD(\ref{eq4.1}-\ref{eq4.3}) to Salt-Pepper Noisy image deblurring. Recovering image  $x \in {R^{mn \times 1}}$ from Salt-Pepper Noisy and blurred observation $b$ can be written as
 \begin{equation}
\mathop {\min }\limits_x \alpha ||Kx - b|{|_1} + ||Dx|{|_{2,1}}
\label{eq5.3}
 \end{equation}
 Here convolution operator $K$ and total difference operator $D$  are the same as section \ref{section5.1}. Setting $K=I$  will reduces above formulation to salt pepper noisy image denoising.
 The saddle point model of (\ref{eq5.3}) with strongly convex dual component can be written as
 \begin{equation}
 \mathop {\min }\limits_x \left\{ {\alpha ||Kx - b|{|_1} + \mathop {\max }\limits_y \left( {\left\langle {Dx,y} \right\rangle  - \frac{{{\mu _g}}}{2}||y|{|^2}} \right)} \right\}
 \label{eq5.4}
  \end{equation}
 and is equivalent to
 \begin{equation}
 \begin{array}{l}
 \mathop {\min }\limits_x \mathop {\max }\limits_u \alpha \left\langle {Kx - b,u} \right\rangle  + \mathop {\max }\limits_v \left\langle {Dx,v} \right\rangle  - \frac{{{\mu _g}}}{2}||v|{|^2}\\
 s.t.\begin{array}{*{20}{c}}
 {}
 \end{array}||u|{|_\infty } \le 1,\begin{array}{*{20}{c}}
 {}
 \end{array}||{v_i}|{|_2} \le 1
 \end{array}
 \end{equation}
 which is a special case of saddle point model (\ref{eq1.1}) with 
\[\begin{array}{l}
f(x) = 0,\begin{array}{*{20}{c}}
{}
\end{array}A = \left( {\begin{array}{*{20}{c}}
	D\\
	{\alpha K}
	\end{array}} \right),\begin{array}{*{20}{c}}
{}
\end{array}y = \left( {\begin{array}{*{20}{c}}
	v\\
	u
	\end{array}} \right)\\
g(y) = {\delta _{\cal U}}(u) + \left\langle {\alpha b,u} \right\rangle  + {\delta _{\cal V}}(v) + {\mu _g}||y|{|^2}/2
\end{array}\]
Here ${\delta _{\cal U}}(u)$  and ${\delta _{\cal V}}(v)$
  denote respectively the indicator function of convex set
 \[{\cal U} = \{ u|u \in {R^{mn \times 1}},||u|{|_\infty } \le 1\} \]
 and
 \[{\cal V} = \{ v|\begin{array}{*{20}{c}}
 v
 \end{array} \in {R^{2mn \times 1}},||{v_i}|{|_2} \le 1\begin{array}{*{20}{c}}
 {,{v_i} \in {R^2},\begin{array}{*{20}{c}}
 	{i = 1,...,mn}
 	\end{array}}
 \end{array}\} \]
 Update of dual variable $y = [v;u]$
   is split into update of block variable $v$  and $u$  . Update of   $v$ proceeds as
   ${v_{t + 1}} = {P_{|| \cdot |{|_2} \le 1}}\left[ {({{\bar v}_t} + {\tau _t}D{x_t})/{\rm{(}}{\tau _t}{\mu _g}{\rm{ + 1)}}} \right]$
  which is a scaled gradient ascent and projection onto Euclidean Ball. Variable $u$  is updated by
     ${u_{t + 1}} = {P_{[ - 1,1]}}({\bar u_t} + {\tau _t}\alpha (K{x_t} - b))$ which contains a gradient ascent and projection onto $[-1,1]$  by 
     ${P_{[ - 1,1]}}$ . While update for primal variable $x$ proceeds as 
     ${x_{t + 1}} = {x_t} - {\eta _t}({D^T}{v_{t + 1}} + \alpha {K^T}{u_{t + 1}})$.
 
   \section{Experimental Reports}
   In this section, we made the comparison on noisy image deblurring to show the acceleration of LDPD and EDPD. Experiments were conducted in Windows 7 and MATLAB 7.1 on a computer with dual core 1.9GHZ CPU and 8GB memory. Image quality is measured by signal to noise ratio (SNR) in in decibel(dB) as
  SNR$({x_k}) = 20{\log _{10}}\left( {||{x^*} - \bar x||/||{x^*} - {x_k}||} \right)$
     . Here ${x^*}$ denotes original image,  
    $\bar x$  the average value of 
    ${x^*}$, and ${x_k}$  the recovered after $k$ iterations.

   \subsection{Gaussian Noisy Image Deblurring } 

  In this section, we compare three Linearized DPD(LDPD) methods including full acceleration by strongly convex dual component in section \ref{section3.3} with the rate
   $O([{L_f} + {\mu _g} + ||A||]/{k^2})$ , partial acceleration in section \ref{section3.2} with the rate $O({L_f}/{k^2} + ||A||/k)$  and single step gradient method in Remark \ref{remark3.1} with rate $O({L_f}/k + ||A||/k)$.

   Penalty $\mu$ in (\ref{eq5.2}) is set to 3000, while dual strongly convexity parameter is fixed as ${\mu _g} = 0.01$.
Using H =fspecial('motion',30,90+45)and imfilter(),motion blur with len=30 and theta=135 was imposed to image Lena of 256 by 256,
   and Gaussian noise of zero mean and standard deviation 3e-3 was added in by imnoise().All compared methods are terminated after 200 iterations.
   
   \begin{figure}[htb]
   	\centering
   	\label{fig:a}\includegraphics[scale=0.8]{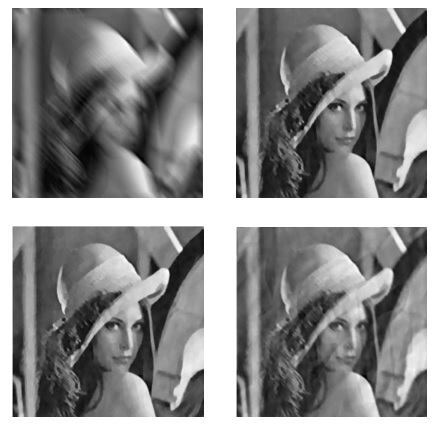}
   	\caption{Gaussian Noisy Image with Motion Blur and Recovery by LDPD}
   	\label{fig:deblurring1}
   \end{figure}
\begin{figure}[htb]
	\centering
	\includegraphics[scale=0.6]{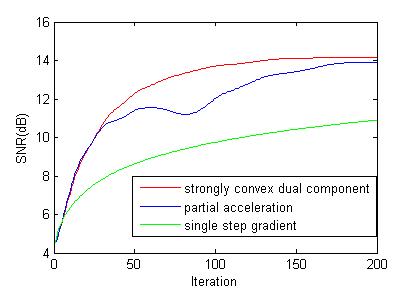}
	\caption{SNR History of Gaussian Noisy Image Deblurring by LDPD}
	\label{fig:snr1}
\end{figure}

In Figure 1, the left top is motion blurred Lenna image with guassian noise, the right top is recovered by LDPD with strongly convex dual component and obtains SNR14.14dB, the left bottom is restored by LDPD with partial acceleration while achieving SNR 13.91dB, the right bottom is computed by LDPD with single step gradient method and its SNR was 10.91dB. 

\indent Fig.2  illustrates the convergence behavior of above three methods, apparently LDPD with strongly convex dual component achieved full acceleration, while  partial acceleration was achieved by imposing multi-step gradient update on primal component. LDPD with single step gradient method lagged behind since it did not utilize any saddle point structure.

\subsection{Salt-Pepper Noisy Image Deblurring}
This test compare the performance of two Exact DPD methods including EDPD with strongly convex dual component in section \ref{section4.4} at the rate $O([{\mu _g} + ||A||]/{k^2})$, and EDPD without
strongly convex dual component in Remark \ref{remark4.1} at the rate $O(||A||/k)$.  Penalty $\alpha$ in (\ref{eq5.4}) is set to 4 while dual strongly convexity parameter
${\mu _g}$ is given the initial value $0.03$  and decreased by half after each 10 iterations. Using H = fspecial('average',5) and imfilter(), averaging blur with size 5 by 5 was imposed on Image man of size 256 by 256, and $20\%$ Salt-Pepper noise was added in. 

\begin{figure}[htb]
	\centering
	\includegraphics[scale=0.8]{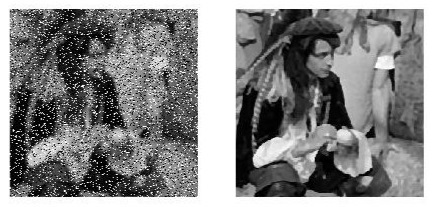}
	\caption{Salt Pepper Noisy Image with Average Blur and Recovery by EDPD.}
	\label{fig:man1}
\end{figure}
In Fig.3 ,the left is the blurred image corrupted by salt pepper noise, and
the right is restored by EDPD.
\begin{figure}[htb]
	\centering
	\includegraphics[scale=0.8]{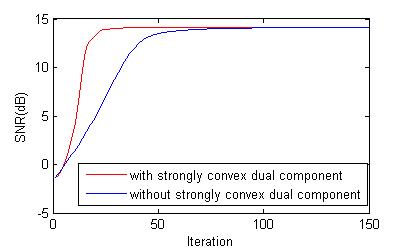}
	\caption{SNR History of Salt Pepper Noisy Image Deblurring by EDPD.}
	\label{fig:snr2}
\end{figure}
 More precisely, EDPD with or without strongly convex dual component achieved visual quality in terms of SNR 14.15dB, 14.13 respectively after 150 iterations. Though they obtained almost the same SNR, the SNR history  in Fig.4 shows the faster convergence behaviour of EDPD with strongly convex dual component over ordinary EDPD without strongly convex component.
 
 \section{Discussions}
 This paper presents a simple primal-dual proximal method named DPD which has linearization version LDPD and exact version EDPD. Both methods allows full acceleration by exploiting strong convexity of primal or dual component. Experiments on Gaussian or Salt Pepper noisy image deblurring verifies the acceleration of presented method.

\section*{Acknowledgements}
\emph{Thanks to arxiv.org for posting this manuscript.}

\bibliographystyle{unsrt}
\bibliography{ref}
\end{spacing}
\end{document}